\documentclass[a4paper,12pt,leqno]{article}
\usepackage[margin=2.5cm]{geometry}
\usepackage[all]{xy}
\usepackage{colonequals}
\usepackage{amssymb}
\usepackage{amsxtra}
\usepackage{amsthm}
\usepackage[T1]{fontenc}
\usepackage{lmodern}
\usepackage[pagebackref,breaklinks]{hyperref}
\usepackage[lite,abbrev,shortalphabetic]{amsrefs}
\theoremstyle{definition}
\newtheorem{definition}[subsection]{Definition}
\newtheorem{remark}[subsection]{Remark}
\newtheorem{notation}[subsection]{Notation}
\theoremstyle{plain}
\newtheorem{lemma}[subsection]{Lemma}
\newtheorem{prop}[subsection]{Proposition}
\newtheorem{theorem}[subsection]{Theorem}
\newtheorem{cor}[subsection]{Corollary}
\numberwithin{equation}{section}

\newcommand{\Z}{\mathbb{Z}}
\newcommand{\Q}{\mathbb{Q}}
\newcommand{\R}{\mathbb{R}}
\newcommand{\C}{\mathbb{C}}
\newcommand{\F}{\mathbb{F}}
\newcommand{\A}{\mathbb{A}}

\newcommand{\cA}{\mathcal{A}}

\newcommand{\cD}{\mathcal{D}}
\newcommand{\cE}{\mathcal{E}}
\newcommand{\cF}{\mathcal{F}}
\newcommand{\cG}{\mathcal{G}}
\newcommand{\cH}{\mathcal{H}}

\newcommand{\cO}{\mathcal{O}}
\newcommand{\cP}{\mathcal{P}}
\newcommand{\cQ}{\mathcal{Q}}

\newcommand{\cHom}{\mathcal{H}\mathnormal{om}}
\newcommand{\op}{\mathrm{op}}

\newcommand{\Cat}{\mathcal{C}\mathnormal{at}}
\newcommand{\Ext}{\mathrm{Ext}}
\newcommand{\Gal}{\mathrm{Gal}}

\newcommand{\alg}{\mathrm{alg}}
\newcommand{\mix}{\mathrm{m}}

\newcommand{\mot}{\mathrm{mot}}

\newcommand{\Perv}{\mathrm{Perv}}
\newcommand{\FEt}{\mathrm{FEt}}
\newcommand{\Fin}{\mathrm{Fin}}
\newcommand{\Stk}{\mathrm{Stk}}
\newcommand{\rss}{\mathrm{ss}}
\newcommand{\Vect}{\mathrm{Vect}}
\newcommand{\Aut}{\mathrm{Aut}}
\newcommand{\Ker}{\mathrm{Ker}}

\newcommand{\GL}{\mathrm{GL}}
\newcommand{\Hom}{\mathrm{Hom}}
\newcommand{\Spec}{\mathrm{Spec}}

\newcommand{\Sh}{\mathrm{Shv}}
\newcommand{\Fr}{\mathrm{Frob}}

\newcommand{\lisse}{\mathrm{lisse}}

\newcommand{\tr}{\mathrm{tr}}

\newcommand{\for}{\underline{\mathrm{for}}}
\renewcommand{\lim}{\underline{\mathrm{lim}}}

\newcommand{\simto}{\xrightarrow{\sim}}
\newcommand{\Fqb}{\overline{\F_q}}
\newcommand{\Zb}{\overline{\Z}}
\newcommand{\Qb}{\overline{\Q}}
\newcommand{\Qlb}{\overline{\Q_\ell}}
\newcommand{\Zlb}{\overline{\Z_\ell}}
\newcommand{\Qlbp}{\overline{\Q_{\ell'}}}
\newcommand{\res}{\mathbin{|}}
\newcommand{\bIsom}{\mathbf{Isom}}
\newcommand{\bHom}{\mathbf{Hom}}

\begin{document}
\title{Companions on Artin stacks}
\author{Weizhe Zheng\thanks{Morningside Center of Mathematics and Hua Loo-Keng Key Laboratory
of Mathematics, Academy
of Mathematics and Systems Science, Chinese Academy of Sciences, Beijing
100190, China; University of the Chinese Academy of Sciences, Beijing
100049, China; email: \texttt{wzheng@math.ac.cn}. Partially supported by
China's Recruitment Program of Global Experts; National Natural Science
Foundation of China Grants 11321101, 11621061, 11688101; National Center for
Mathematics and Interdisciplinary Sciences, Chinese Academy of Sciences.}
}
\date{}
\maketitle

\begin{abstract}
Deligne's conjecture that $\ell$-adic sheaves on normal schemes over a
finite field admit $\ell'$-companions was proved by L.~Lafforgue in the
case of curves and by Drinfeld in the case of smooth schemes. In this
paper, we extend Drinfeld's theorem to smooth Artin stacks and deduce
Deligne's conjecture for coarse moduli spaces of smooth Artin stacks. We
also extend related theorems on Frobenius eigenvalues and traces to Artin
stacks.
\end{abstract}

\section*{Introduction}
Let $\F_q$ be a finite field and let $\ell$ and $\ell'$ be prime numbers not
dividing $q$. We let $\Qlb$ denote an algebraic closure of $\Q_\ell$.
Deligne conjectured \cite[Conjecture 1.2.10]{WeilII} that every lisse
$\Qlb$-sheaf on a normal scheme separated of finite type over $\F_q$ admits
a lisse $\overline{\Q_{\ell'}}$-companion. Drinfeld \cite[Theorem 1.1]{Dr}
proved this conjecture for smooth schemes. The goal of this paper is to
extend Drinfeld's theorem to smooth Artin stacks. We deduce that Deligne's
conjecture holds for coarse moduli spaces of smooth Artin stacks. We also
extend related theorems on Frobenius eigenvalues and traces to Artin stacks.

For an Artin stack $X$ of finite presentation over $\F_q$ and a Weil
$\Qlb$-sheaf $\cF$ on $X$, we let $E(\cF)$ denote the subfield of $\Qlb$
generated by the local Frobenius traces $\tr(\Fr_x, \cF_{\bar{x}})$, where
$x\in X(\F_{q^n})$ and $n\ge 1$. Here $\Fr_x=\Fr_{q^n}$ denotes the
geometric Frobenius, and $\bar x$ denotes a geometric point above $x$. Let
$E_{\lambda'}$ be an algebraic extension of $\Q_{\ell'}$ and let
$\sigma\colon E(\cF)\to E_{\lambda'}$ be a field embedding, \emph{not}
necessarily continuous. We say that a Weil $E_{\lambda'}$-sheaf $\cF'$ is a
\emph{$\sigma$-companion} of $\cF$ if for all $x\in X(\F_{q^n})$ with $n\ge
1$, we have $\tr(\Fr_x,\cF'_{\bar{x}})=\sigma \tr(\Fr_x,\cF_{\bar{x}})$.

Our main results on Frobenius eigenvalues and traces are as follows.

\begin{theorem}\label{t.2}
Let $X$ be a geometrically unibranch\footnote{For a short review of the
property ``geometrically unibranch'', see Remark \ref{r.unib}.} Artin stack
of finite presentation over $\F_q$ and let $\cF$ be a simple lisse
$\Qlb$-sheaf of rank $r$ on $X$ such that $\det(\cF)$ has finite order.
\begin{enumerate}
\item (Frobenius eigenvalues) Let $x\in X(\F_{q^n})$ and let $\alpha$ be
    an eigenvalue of $\Fr_x$ acting on $\cF_{\bar x}$. Then $\alpha$ is a
    $q$-Weil number of weight $0$.\footnote{We adopt the convention that a
    \emph{$q$-Weil number of weight $0$} is an algebraic number $\alpha$
such that for every place $\lambda$ of $\Q(\alpha)$ not dividing $q$
(finite or Archimedean), we have $\lvert \alpha\rvert_\lambda=1$.}
Moreover, for
    every valuation $v$ on $\Q(\alpha)$ such that $v(q^n)=1$, we have
    $\lvert v(\alpha)\rvert \le \frac{r-1}{2}$.
\item (Frobenius traces) The field $E(\cF)$ is a number field (namely, a
    finite extension of $\Q$).
\end{enumerate}
\end{theorem}

The statement of Theorem \ref{t.2}, with a slightly weaker bound for the
$p$-adic valuations, is conjectured to hold for normal schemes separated of
finite type over $\F_q$ by Deligne \cite[Conjecture 1.2.10
(i)--(iv)]{WeilII}. In the case of curves, the theorem with a weaker bound
is a consequence of the Langlands correspondence for $\GL(n)$ over function
fields proved by L.~Lafforgue \cite[Th\'eor\`eme VII.6]{Lafforgue}. The
improvement of the bound is due to V.~Lafforgue \cite[Corollaire
2.2]{Lafforgue}. The extension from curves to schemes is stated by L.
Lafforgue \cite[Proposition VII.7]{Lafforgue} for part (1), and due to
Deligne \cite[Th\'eor\`eme 3.1]{Deligne} for part (2).

Recently Drinfeld and Kedlaya \cite[Theorem 1.3.3]{DK} proved a refinement
of V.~Lafforgue's bound for Newton polygons, which can be thought of as an
analogue of Griffiths transversality. We also extend this result to stacks
(Theorem \ref{t.slopes}).

The following is our main result on companions.

\begin{theorem}[Companions on smooth stacks]\label{t.1}
Let $X$ be a smooth Artin stack over $\F_q$ of finite presentation and
separated diagonal. Let $\cF$ be a lisse Weil $\Qlb$-sheaf on $X$. Then, for
every embedding $\sigma\colon E(\cF)\to \Qlbp$, $\cF$ admits a lisse
$\sigma$-companion $\cF'$. Moreover, if $E(\cF)$ is a number field, then
there exists a finite extension $E$ of $E(\cF)$ such that for every finite
place $\lambda'$ of $E$ not dividing $q$, $\cF$ admits a lisse
$\sigma_{\lambda'}$-companion. Here $\sigma_{\lambda'}\colon E(\cF)\to E \to
E_{\lambda'}$, and $E_{\lambda'}$ denotes the completion of $E$ at
$\lambda'$.
\end{theorem}

The statements of Theorem \ref{t.1} are conjectured to hold for normal
schemes separated of finite type over $\F_q$ by Deligne \cite[Conjecture
1.2.10 (v)]{WeilII}. In the case of curves, the first assertion of the
theorem is due to L.~Lafforgue \cite[Th\'eor\`eme VII.6]{Lafforgue}, and the
second to Chin \cite{Chin}. The extension from curves to smooth schemes is
due to Drinfeld \cite[Theorem 1.1]{Dr}.

As an application of Theorem \ref{t.1}, we deduce that Deligne's conjecture
holds for coarse moduli spaces of smooth Artin stacks.

\begin{cor}[Companions on coarse moduli spaces]\label{c.moduli}
Let $X$ be a scheme or algebraic space that is Zariski locally the coarse
moduli space of a smooth Artin stack of finite inertia and finite
presentation over $\F_q$ (e.g.\ when $X$ has quotient singularities). Then
the statements of Theorem \ref{t.1} hold for $X$.
\end{cor}

The general normal case seems difficult. Drinfeld deduces his result from an
equivalence between lisse sheaves on a regular scheme $X$ and compatible
systems of lisse sheaves on curves on $X$ \cite[Theorem 2.5]{Dr}. This
equivalence fails for $X$ normal in general \cite[Section 6]{Dr}.

For both Theorems \ref{t.2} and \ref{t.1}, we reduce first to the case of a
quotient stack $[Y/G]$ of a quasi-projective scheme $Y$ by a finite group
$G$, and then, choosing an embedding $G\to \GL_m$, to the scheme $Y\wedge^G
\GL_m$. One step of the reduction consists of showing that the assertions
can be checked on any dense open substack.

Gabber's theorem \cite{Gabber} that companionship is preserved by operations
on the Grothendieck groups extends to stacks \cite[Proposition 5.8]{Zind}.
Combining this with Theorem \ref{t.1}, one obtains the existence of perverse
companions on (not necessarily smooth) Artin stacks.

The paper is organized as follows. In Section \ref{s.1}, we establish some
preliminary results on Weil sheaves. In Section \ref{s.1t}, we prove
Theorems \ref{t.2} (1) on Frobenius eigenvalues and Theorem \ref{t.slopes}
on Newton polygons. We deduce from Theorem \ref{t.2} (1) that the bounded
derived category $D^b_c(X,\Qlb)$ is a direct sum of twists of the derived
category of weakly motivic complexes for any Artin stack $X$ of finite
presentation over $\F_q$. In Section \ref{s.2}, we prove Theorem \ref{t.2}
(2) on Frobenius traces. In Section \ref{s.3}, we prove Theorem \ref{t.1}
and Corollary \ref{c.moduli} on lisse companions. We deduce results on
perverse companions and companions in Grothendieck groups on Artin stacks of
finite presentation and separated diagonal over $\F_q$. In an appendix
(Section \ref{s.4}), we prove that pure perverse sheaves on $X$ are
geometrically semisimple, without assuming that the stabilizers are affine,
extending a result of Sun \cite[Theorem 3.11]{Sun}.

Unless otherwise stated, all stacks are assumed to be Artin stacks of finite
presentation over $\F_q$, not necessarily of separated diagonal, and sheaves
are assumed to be constructible. We write $D(X,\Qlb)$ for $D_c(X,\Qlb)$. We
will only consider the middle perversity.

\subsection*{Acknowledgments}
This paper grows out of an answer to Shenghao Sun's question of extending
the theorems of Deligne and Drinfeld to stacks. I thank Yongquan Hu, Yifeng
Liu, Martin Olsson, and Shenghao Sun for useful discussions, and Vladimir
Drinfeld and Luc Illusie for valuable comments. I am grateful to Ofer Gabber
for pointing out a mistake in a draft of this paper. I thank the referee for
a careful reading of the manuscript and many helpful comments. Part of this
paper was written during a stay at Shanghai Center for Mathematical Sciences
and I thank the center for hospitality.

\section{Weil sheaves}\label{s.1}
For problems concerning companions, it is convenient to work with Weil
sheaves. In this section, we establish some preliminary results on Weil
sheaves. The main result is Proposition \ref{p.twist} on the determinant of
lisse Weil $\Qlb$-sheaves on geometrically unibranch stacks. We deduce that
the category of Weil $\Qlb$-sheaves is a direct sum of the twists of the
category of $\Qlb$-sheaves (Proposition \ref{p.Weil}).

Let $E_\lambda$ be an algebraic extension of $\Q_\ell$. A \emph{Weil
$E_\lambda$-sheaf} on a stack $X$ is an $E_\lambda$-sheaf $\cF$ on
$X\otimes_{\F_q} \Fqb$ equipped with an action of the Weil group
$W(\Fqb/\F_q)$ lifting the action of $W(\Fqb/\F_q)$ on $X\otimes_{\F_q}
\Fqb$. A morphism of Weil $E_\lambda$-sheaves on $X$ is a morphism of the
underlying $E_\lambda$-sheaves on $X\otimes_{\F_q} \Fqb$ compatible with the
action of $W(\Fqb/\F_q)$.

\begin{remark}\label{r.2lim}
More formally, the category $\Sh^W(X,E_\lambda)$ of Weil $E_\lambda$-sheaves
on $X$ is a (pseudo)limit of the diagram (i.e.\ pseudofunctor) $B\Z\to \Cat$
given by the action of $\Z\simeq W(\Fqb/\F_q)$ on the category
$\Sh(X\otimes_{\F_q}\Fqb,E_\lambda)$, where $B\Z$ is the groupoid associated
to the group $\Z$ and $\Cat$ is the $2$-category of categories. If we let
$\Cat^{B\Z}$ denote the $2$-category of diagrams $B\Z\to \Cat$, the
forgetful $2$-functor $\for\colon \Cat^{B\Z}\to \Cat$ and the limit
$2$-functor $\lim\colon \Cat^{B\Z}\to \Cat$ preserve limits (up to
equivalences).
\end{remark}

\begin{remark}\label{r.ext}
The functor $\Sh(X,E_\lambda)\to \Sh^W(X,E_\lambda)$ carrying $\cF$ to
$(\cF_{\Fqb} ,\phi)$, where $\cF_{\Fqb}$ is the pullback of $\cF$ to
$X\otimes_{\F_q}\Fqb$ and $\phi$ is the restriction of the action of
$\Gal(\Fqb/\F_q)$ to $W(\Fqb/\F_q)$, is fully faithful. Moreover, its
essential image is stable under extension by the following general facts on
extensions (cf.\ \cite[Proposition 5.1.2]{BBD}).
\begin{enumerate}
\item Let $(\cA,F)$ be an Abelian category with $\Z$-action (i.e.\ a
    pseudofunctor from $B\Z$ to the $2$-category of Abelian categories).
    For objects $(A,\phi)$ and $(B,\psi)$ of the limit category $\cA^F$,
    we have a short exact sequence of Abelian groups (cf.\ \cite[page
    124]{BBD})
    \[0\to \Hom_{\cA}(A,B)_\Z\to \Ext^1_{\cA^F}((A,\phi),(B,\psi))\to \Ext^1_\cA(A,B)^\Z\to 0.\]
\item Let $\cD$ be a triangulated category equipped with a $t$-structure.
    Note that $\cD$ is not necessarily equivalent to the derived category
    of its heart $\cA$. Nonetheless, for $A$ and $B$ in $\cA$, we have an
    isomorphism $\Hom_{\cD}(A,B[1])\simeq \Ext^1_{\cA}(A,B)$ carrying $f$
    to the extension given by completing $f$ into a distinguished triangle
    (cf.\ \cite[Remarque 3.1.17 (ii)]{BBD}).
\end{enumerate}
\end{remark}

Recall from \cite[Tag 06U6]{SP} that a morphism $f\colon X\to Y$ of Artin
stacks (not necessarily of finite presentation over $\F_q$) is said to be
\emph{submersive} if the induced continuous map is a quotient map, and
\emph{universally submersive} if for every morphism of Artin stacks $Y'\to
Y$, the base change $X\times_Y Y'\to Y'$ of $f$ is submersive.

\begin{lemma}\label{l.descent}
Let $f\colon X\to Y$ be a universally submersive morphism of stacks. Then
$f$ is of effective descent for Weil $E_\lambda$-sheaves and for
$E_\lambda$-sheaves.
\end{lemma}

The statement of the lemma means that $f^*$ induces an equivalence of
categories from $\Sh^W(Y,E_\lambda)$ to the category of descent data, which
is a limit of the diagram
\[\xymatrix{\Sh^W(X,E_\lambda)\ar@<0.5ex>[r]\ar@<-0.5ex>[r]& \Sh^W(X\times_Y X,E_\lambda)\ar@<1ex>[r]\ar@<-1ex>[r]\ar[r]& \Sh^W(X\times_Y X\times_Y X,E_\lambda).}\]
induced by inverse image functors. An object of the category of descent data
is a Weil $E_\lambda$-sheaf $\cF$ on $X$ endowed with an isomorphism
$p_1^*\cF\simeq p_2^*\cF$ satisfying the cocycle condition. Here
$p_1,p_2\colon X\times_Y X\to X$ are the two projections. Compare with
\cite[Proposition 2.4]{IZ}.

\begin{proof}
Consider the pseudofunctor $F\colon \Stk\to \Cat^{B\Z}$ carrying $X$ to the
diagram in Remark \ref{r.2lim}, where $\Stk$ is the $2$-category of stacks.
The assertion of the lemma for Weil $E_\lambda$-sheaves is that $f$ is of
effective descent for the pseudofunctor $\lim \circ F$. Since $\lim$ and
$\for$ preserve limits, we may replace $\lim\circ F$ by $F$, and then by
$\for\circ F$. In other words, it suffices to show that $f\otimes_{\F_q}
\overline{\F_q}$ is of effective descent for $E_\lambda$-sheaves.

Changing notation, let us show that any universally submersive morphism
$r\colon X\to Y$ of Artin stacks of finite presentation over $\F_q$ or
$\overline{\F_q}$ is of effective descent for $E_\lambda$-sheaves. Let
$f\colon Y'\to Y$ be a smooth presentation with $Y'$ a scheme and let $X'\to
X\times_Y Y'$ be a smooth presentation with $X'$ a scheme. Consider the
square
\[\xymatrix{X'\ar[r]^u\ar[d]_{g} & Y'\ar[d]^{f}\\
X\ar[r]^{r} & Y.}
\]
By a general property of descent \cite[Th\'eor\`eme 10.4, line 6 of the
table]{Giraud}, it suffices to show that $u$, $f$, $g$, and $g'\colon
X'\times_{Y'} X'\to X\times_Y X$ are of effective descent. Note that $f$,
$g$, and $g'$ are representable and smooth surjective, and $u$ is a
universally submersive morphism of schemes. By a theorem of Voevodsky
\cite[Theorem 3.1.9]{Voe}, $u$ is dominated by $\coprod V_i \to V
\xrightarrow{v} Y'$, where $v$ is a proper surjective morphism of schemes
and $(V_i)$ is a finite Zariski open cover of $V$. By general properties of
descent \cite[Proposition 6.25 (ii), (iii)]{Giraud}, we are reduced to
showing that $r$ is of effective descent in the following cases: (a) $r$ is
representable and smooth surjective; (b) $r$ is a proper surjective morphism
of schemes. This follows from Beck's theorem \cite[Proposition 5]{BR} (cf.\
\cite[VIII 9.4.1]{SGA4}). Indeed, $r^*$ is exact and conservative, and the
Beck-Chevalley condition is verified by smooth base change in case (a) and
proper base change in case (b).
\end{proof}

A Weil $E_\lambda$-sheaf $\cF$ on a stack $X$ is called \emph{lisse} if
there exists a smooth presentation $f\colon Y\to X$ such that the pullback
of $\cF$ to $Y\otimes_{\F_q} \overline{\F_q}$ is isomorphic to
$\cG\otimes_\cO E_\lambda$ for a lisse $\cO$-sheaf $\cG$, where $\cO$ is the
ring of integers of a finite extension of $\Q_\ell$ in $E_\lambda$. Lemma
\ref{l.descent} also holds for lisse Weil $E_\lambda$-sheaves (and lisse
$E_\lambda$-sheaves). This follows from the lemma and the following fact.

\begin{lemma}\label{l.lisse}
Let $f\colon X\to Y$ be a universally submersive morphism of stacks and let
$\cF$ be a Weil $E_\lambda$-sheaf on $Y$. Then $\cF$ is lisse if and only if
$f^*\cF$ is lisse.
\end{lemma}

\begin{proof}
The ``only if'' part is trivial. To show the ``if'' part, by taking smooth
presentations, we are reduced to the case of schemes (over $\Fqb$) and
$\cO$-sheaves. In this case, the assertion follows from the fact that $f$ is
of effective descent for \'etale morphisms \cite[Theorem 5.17]{Rydh}.
\end{proof}

\begin{remark}\label{r.unib}
We say that a stack $X$ is \emph{geometrically unibranch} if for some (or,
equivalently, for every) smooth presentation $Y\to X$, the strict
localizations of $Y$ are irreducible. Normal stacks are geometrically
unibranch. If $X$ is geometrically unibranch, then every lisse Weil
$E_\lambda$-sheaf $\cF$ satisfies $j_*j^*\cF\simeq \cF\otimes j_*
E_\lambda\simeq \cF$ for every dominant open immersion $j\colon U\to X$. It
follows that the pullback of $\cF$ to $X\otimes_{\F_q} \overline{\F_q}$
comes from a lisse $\cO$-sheaf.

Let $X$ be a connected stack and let $\bar x\to X$ be a geometric point. The
fundamental group $\pi_1(X,\bar x)$ is defined in \cite[Section 4]{Noohi},
which extends to stacks not necessarily of separated diagonal, as the group
of automorphisms of the functor $\FEt(X)\to \Fin$ carrying $Y\to X$ to (the
underlying set of) the geometric fiber $Y\times_X \bar x$. Here $\FEt(X)$
denotes the category of finite\footnote{Recall that a morphism of stacks is
said to be \emph{finite} if it is representable by schemes and finite.}
\'etale morphisms over $X$, and $\Fin$ denotes the category of finite sets.
We define the Weil group $W(X,\bar x)$ to be the inverse image of
$W(\overline{\F_q}/\F_q)$ under the homomorphism $\pi_1(X,\bar x)\to
\Gal(\overline{\F_q}/\F_q)$. The functor $\Sh^W_\lisse(X,E_\lambda)\to
\Vect(E_\lambda)$ carrying $\cF$ to its stalk $\cF_{\bar x}$ at $\bar x$ is
conservative.

If $X$ is connected and geometrically unibranch, then $X$ is irreducible and
Weil $E_\lambda$-sheaves (resp.\ $E_\lambda$-sheaves) on $X$ correspond to
$E_\lambda$-representations of the Weil (resp.\ fundamental) group of $X$.
\end{remark}

\begin{remark}\label{r.Qlb}
Let $X$ be a stack. For Weil $E_\lambda$-sheaves $\cF$ and $\cG$ on $X$, we
have $\cF\simeq \cG$ if and only if $\cF\otimes_{E_\lambda}\Qlb\simeq
\cG\otimes_{E_\lambda}\Qlb$. Similarly, for $A$ and $B$ in the bounded
derived category $D^b(X,E_\lambda)$ of $E_\lambda$-sheaves, $A\simeq B$ if
and only if $A\otimes_{E_\lambda}\Qlb\simeq B\otimes_{E_\lambda}\Qlb$. This
follows from Lemma \ref{l.Zariski} below and the fact that rational points
form a Zariski dense subset of any affine space over an infinite field (here
$E_\lambda$).

Moreover, a Weil $E_\lambda$-sheaf $\cF$ on $X$ is an $E_\lambda$-sheaf if
and only if the Weil $\Qlb$-sheaf $\cF\otimes_{E_\lambda}\Qlb$ is a
$\Qlb$-sheaf. Indeed, we reduce to the case of schemes by Lemma
\ref{l.descent} and then to lisse sheaves on irreducible geometrically
unibranch schemes by Remark \ref{r.ext}. In this case, the assertion is
clear, as the Weil group is dense in the fundamental group.

For these reasons, we will work mostly with Weil $\Qlb$-sheaves rather than
Weil $E_\lambda$-sheaves.
\end{remark}

The following is a variant of \cite[Lemma 2.1.3]{SZ}.

\begin{lemma}\label{l.Zariski}
Let $X$ be a stack and let $A$ and $B$ be Weil $E_\lambda$-sheaves on $X$
(resp.\ $A,B\in D^b(X,E_\lambda)$). Then there exists a Zariski open
subscheme $U=\bIsom(A,B)$ of the affine space $\bHom(A,B)$ over $E_\lambda$
represented by the $E_\lambda$-vector space $\Hom(A,B)$ such that for any
algebraic extension $E'_\lambda$ of $E_\lambda$, the set $U(E'_\lambda)$ is
the set of isomorphisms $A\otimes_{E_\lambda} E'_\lambda\simto
B\otimes_{E_\lambda} E'_\lambda$.
\end{lemma}

Note that $\Hom(A,B)$ is finite-dimensional. Indeed, we have
\[\Hom(A,B)=\Hom(A_{\Fqb},B_{\Fqb})^{W(\Fqb/\F_q)}\]
in the case of Weil sheaves, and a
short exact sequence
\[0\to \Hom(A_{\Fqb},B_{\Fqb}[-1])_{\Gal(\Fqb/\F_q)}\to \Hom(A,B)\to \Hom(A_{\Fqb},B_{\Fqb})^{\Gal(\Fqb/\F_q)}\to 0\]
in the case of $D^b(X,E_\lambda)$. Here $A_{\Fqb}$ and $B_{\Fqb}$ denote the
pullbacks of $A$ and $B$ to $X\otimes_{\F_q} \Fqb$. The proof of the lemma
is the same as \cite[Lemma 2.1.3]{SZ}, by taking a finite number of stalk
functors.

Following \cite[1.2.7]{WeilII}, for $a\in \Qlb^\times$, we let $\Qlb^{(a)}$
denote the Weil $\Qlb$-sheaf on $\Spec(\F_q)$ of rank one such that the
geometric Frobenius $\Fr_q\in \Gal(\Fqb/\F_q)$ acts by multiplication by
$a$. For a stack $X$, we still denote $\pi_X^*\Qlb^{(a)}$ by $\Qlb^{(a)}$,
where $\pi_X\colon X\to \Spec(\F_q)$ is the projection. We put
$\cF^{(a)}\colonequals \cF\otimes\Qlb^{(a)}$.

The following is an extension to stacks of Deligne's result on determinants
\cite[Propositions 1.3.4 (i), 1.3.14]{WeilII} (cf.\ \cite[0.4]{Deligne}).

\begin{prop}\label{p.twist}
Let $X$ be an irreducible geometrically unibranch stack. Then, for every
lisse Weil $\Qlb$-sheaf $\cF$ on $X$, there exists $a\in \Qlb^\times$ such
that $\det(\cF^{(a)})$ has finite order. Moreover, every \emph{simple} lisse
Weil $\Qlb$-sheaf $\cF$ on $X$ such that $\det(\cF)$ has finite order is a
$\Qlb$-sheaf.
\end{prop}

Note that $a$ is unique up to multiplication by roots of unity. It follows
from the proposition that every simple lisse Weil $\Qlb$-sheaf is the twist
of some $\Qlb$-sheaf.

Even if we restrict our attention to $\Qlb$-sheaves, the first part of the
proposition is still necessary for the following sections. Our proof of the
proposition relies on Lemma \ref{l.Behrend} below, which will be used in
later sections as well.

\begin{lemma}\label{l.subm}
Let $f\colon X\to Y$ be a universally submersive morphism of stacks with
geometrically connected fibers. Assume $Y$ connected. Then $X$ is connected
and the homomorphism $\pi_1(X)\to \pi_1(Y)$ induced by $f$ is surjective.
\end{lemma}

The case of schemes is \cite[IX Corollaire 5.6]{SGA1}.

\begin{proof}
The first assertion follows from the fact that for any quotient map $X\to Y$
of topological spaces, if $Y$ and the fibers are connected, then $X$ is
connected. The second assertion follows from the first one. Indeed, for any
connected finite \'etale cover $Y'\to Y$, $Y'\times_Y X$ is connected by the
first assertion applied to the projection $Y'\times_Y X\to Y'$.
\end{proof}

\begin{lemma}\label{l.sep}
Any stack $X$ admits a dense open substack with separated diagonal.
\end{lemma}

\begin{proof}
There exists a dense open substack $V$ of $X$ with flat inertia. Then $V$ is
a gerbe over an algebraic space \cite[Tag 06QJ]{SP}. Since any group
algebraic space of finite presentation over a field is separated (and in
fact a group scheme, see \cite[Proposition 5.1.1]{Behrend} or \cite[Tags
08BH, 0B8G]{SP}), there exists a dense open substack $U$ of $V$ with
separated inertia. Then $U$ has separated diagonal by fppf descent
\cite[Tags 0CPS, 0DN6]{SP}.
\end{proof}

\begin{lemma}\label{l.Behrend}
Let $\cF$ be a lisse Weil $\Qlb$-sheaf on a stack $X$. Then there exists a
dominant open immersion $j\colon U\to X$ and a gerbe-like morphism $f\colon
U\to Y$, where $Y$ is a Deligne-Mumford stack, such that the adjunction map
$f^*\cG\to j^*\cF$, where $\cG\colonequals f_*j^*\cF$, is an isomorphism.
\end{lemma}

Recall that any gerbe-like morphism $f\colon U\to Y$ of stacks is a smooth
universal homeomorphism \cite[Tags 06R9, 0DN8]{SP}. Note that $f^*\cG\simeq
j^*\cF$ implies that $\cG$ is a lisse Weil $\Qlb$-sheaf, which is simple if
$j^*\cF$ is simple.  Moreover, if $X$ is geometrically unibranch, then $Y$
is geometrically unibranch, and $\det(\cF)$ and $\det(\cG)$ have the same
(possibly infinite) order by Lemma \ref{l.subm}.

\begin{proof}
By a d\'evissage result of Behrend \cite[Propositions 5.1.11,
5.1.14]{Behrend}, there exists a dominant open immersion $j\colon U\to X$
and a gerbe-like morphism $f\colon U\to Y$, where $Y$ is a Deligne-Mumford
stack, such that the diagonal of $f$ has connected geometric fibers. By
generic base change \cite[Proposition 2.11]{IZ}, up to shrinking $Y$ (and
$U$), we may assume that $f_*\cF$ commutes with base change. It then
suffices to check that the adjunction $f_y^*f_{y*}(\cF\res U_y)\to \cF\res
U_y$ is an isomorphism for every geometric fiber $f_y\colon U_y\to y$ of
$f$. Since $U_y$ is the classifying stack of a connected group scheme over
$y$, any sheaf on $U_y$ is constant and the assertion is trivial.
\end{proof}

\begin{remark}\label{r.gerbe}
Since any gerbe over a finite field is neutral \cite[Corollary
6.4.2]{Behrend}, any point $y\in Y(\F_{q^n})$ lifts to a point of
$U(\F_{q^n})$. In particular, $E(\cG)=E(j^*\cF)$.
\end{remark}

\begin{proof}[Proof of Proposition \ref{p.twist}]
Applying Lemma \ref{l.Behrend}, we are reduced to the case where $X$ is a
Deligne-Mumford stack. Here for the second assertion of the proposition, we
have used the fact that $\cF\simeq j_*j^*\cF$. By Lemma \ref{l.sep} and
\cite[Corollaire 6.1.1]{LMB}, up to shrinking $X$, we may assume $X=[Y/G]$,
where $Y$ is a separated scheme and $G$ is a finite group acting on $Y$. Up
to replacing $Y$ by a connected component and $G$ by the decomposition
group, we may assume that $Y$ is irreducible. Let $g\colon Y\to X$. Then
$g^*\cF$ corresponds to the restriction to the open normal subgroup
$\pi_1(Y)\lhd\pi_1(X)$ of quotient $G$. By the case of schemes of the first
assertion \cite[Proposition 1.3.4 (i)]{WeilII}, there exists $a\in
\Qlb^\times$ such that $\det(g^*\cF^{(a)})^{\otimes n}\simeq \Qlb$ for some
$n\ge 1$. Then $\det(\cF^{(a)})^{\otimes (n\cdot \#G)}\simeq \Qlb $. This
finishes the proof of the first assertion of Proposition \ref{p.twist}. Now
assume that $\cF$ is simple and $\det(\cF)$ has finite order. By Lemma
\ref{l.finite} below and the case of schemes of the second assertion of
Proposition \ref{p.twist} \cite[Proposition 1.3.14]{WeilII}, $g^*\cF$ is a
$\Qlb$-sheaf, so that the same holds for $\cF$ by Lemma \ref{l.descent}.
\end{proof}

\begin{lemma}\label{l.finite}
Let $f\colon X\to Y$ be a finite \'etale morphism of geometrically unibranch
stacks and let $\cF$ be a simple lisse Weil $\Qlb$-sheaf on $Y$ such that
$\det(\cF)$ has finite order. Then $f^*\cF\simeq \bigoplus_i\cF_i$ is
semisimple with simple factors $\cF_i$ such that each $\det(\cF_i)$ has
finite order.
\end{lemma}

\begin{proof}
We may assume that $X$ and $Y$ are irreducible. Since $\pi_1(X)\subseteq
\pi_1(Y)$ is an open subgroup of finite index, $f^*\cF\simeq
\bigoplus_i\cF_i$ is semisimple. For the assertion on simple factors, we may
assume that $f$ is a Galois cover of group $G$. Fix an $i_0$. By the first
assertion of Proposition \ref{p.twist}, there exists $a\in \Qlb^\times$ such
that $\det(\cF_{i_0}^{(a)})$ has finite order. Since the simple factors are
permuted transitively by $G$, $\det(\cF_{i}^{(a)})$ has finite order for
each $i$. It follows that $\det(\cF^{(a)})\simeq
\bigotimes_{i}\det(\cF_i^{(a)})$ has finite order. This implies that $a$ is
a root of unity. Therefore, each $\det(\cF_i)$ has finite order.
\end{proof}

\begin{remark}
\begin{enumerate}
\item In the above proof of Proposition \ref{p.twist}, we first reduce to
    the case of quotient stacks $[Y/G]$ by finite groups $G$ by Lemma
    \ref{l.Behrend}, and then reduce to the case of schemes. The same
    strategy will be used for Theorems \ref{t.2}, \ref{t.1}, and
    \ref{t.slopes}. Another approach is to use the generic existence of a
    smooth presentation with geometrically connected fibers
    (\cite[Th\'eor\`eme 6.5]{LMB} and Lemma \ref{l.sep}) to reduce
    directly to the case of schemes, which works for Proposition
    \ref{p.twist} and Theorems \ref{t.2} (1) and \ref{t.slopes}, but fails
    for Theorems \ref{t.2} (2) and \ref{t.1}.

\item The reduction from quotient stacks $[Y/G]$ to schemes here and in
    Theorem \ref{t.2} (1) uses Lemma \ref{l.finite} applied to the finite
    \'etale cover $Y\to [Y/G]$. We may replace this by Lemma \ref{l.subm}
    applied to the $\GL_m$-torsor $Y\wedge^G \GL_m\to [Y/G]$, where $G\to
    \GL_m$ is a chosen embedding, making the proofs closer to those of
    Theorems \ref{t.2} (2), \ref{t.1}, and \ref{t.slopes}.

\item We can also prove Proposition \ref{p.twist} directly by imitating
    the proof of the case of schemes. Indeed, as in \cite[Proposition
    1.3.4, Variante]{WeilII}, the first assertion follows from the case of
    curves by joining by curves (see the proof of Proposition
    \ref{p.eigen}) and Chebotarev's density theorem (Proposition
    \ref{p.inj}). As in \cite[Theorem 1.3.8]{WeilII}, this implies that
    the radical of $G^{00}$ is unipotent, where $G^{00}$ is the identity
    component of the geometric monodromy group, which is a theorem of
    Grothendieck in the case of schemes. Finally, as in \cite[Proposition
    1.3.14]{WeilII}, the second assertion follows from this and the first
    assertion.
\end{enumerate}
\end{remark}

Proposition \ref{p.twist} has the following consequence on the structure of
Weil $\Qlb$-sheaves. For a stack $X$, we let $\Sh(X,\Qlb)^{(a)}\subseteq
\Sh^W(X,\Qlb)$ denote the full subcategory spanned by Weil $\Qlb$-sheaves of
the form $\cF^{(a)}$ with $\cF\in\Sh(X,\Qlb)$. The subcategory only depends
on the class of $a$ in $\Qlb^\times/\Zlb^\times$, where $\Zlb$ denotes the
ring of integers of $\Qlb$.

\begin{prop}\label{p.Weil}
Let $X$ be a stack. We have a canonical decomposition:
\[\Sh^W(X,\Qlb)\simeq \bigoplus_{a\in \Qlb^\times/\Zlb^\times} \Sh(X,\Qlb)^{(a)}.\]
\end{prop}

\begin{proof}
It suffices to show the following:
\begin{itemize}
\item (generation) Every object of $\Sh^W(X,\Qlb)$ is a successive
    extension of objects of $\Sh(X,\Qlb)^{(a)}$;
\item (orthogonality) $\Ext^i(A^{(a)},B^{(b)})=0$ for $A,B\in
    \Sh(X,\Qlb)$, $a/b\not\in \Zlb^\times$ and $i=0,1$.
\end{itemize}
The first point follows from Proposition \ref{p.twist}. Let us show the
orthogonality. We have
$\Hom(A^{(a)},B^{(b)})=\Hom(A^{(a)}_{\Fqb},B^{(b)}_{\Fqb})^{W(\Fqb/\F_q)}$
and a short exact sequence (Remark \ref{r.ext} (1), (2))
\[0\to \Hom(A^{(a)}_{\Fqb},B^{(b)}_{\Fqb})_{W(\Fqb/\F_q)} \to \Ext^1(A^{(a)},B^{(b)})\to \Hom(A^{(a)}_{\Fqb},B^{(b)}_{\Fqb}[1])^{W(\Fqb/\F_q)}\to 0.\]
The $\Qlb$-vector space $\Hom(A^{(a)}_{\Fqb},B^{(b)}_{\Fqb}[i])$ with
$W(\Fqb/\F_q)$-action can be identified with the Weil $\Qlb$-sheaf
$(R^i\pi_{X*}R\cHom(A,B))^{(b/a)}$ on $\Spec(\F_q)$, where $\pi_X\colon X\to
\Spec(\F_q)$. The eigenvalues of $\Fr_q$ are all in the class of $b/a$, so
the action has no nonzero invariants or coinvariants. Therefore,
$\Ext^i(A^{(a)},B^{(b)})=0$ for $i=0,1$.
\end{proof}

\begin{cor}
Let $X$ be a stack. A Weil $\Qlb$-sheaf $\cF$ is a $\Qlb$-sheaf if and only
if for every $x\in X(\F_{q^n})$, $n\ge 1$, the eigenvalues of $\Fr_x$ acting
on $\cF_{\bar x}$ are $\ell$-adic units.
\end{cor}

By reducing to curves (see Proposition \ref{p.eigen} below), we see that for
lisse Weil $\Qlb$-sheaves, it suffices to check the condition in the
corollary for one given $x$ in each connected component of $X$.

\section{Frobenius eigenvalues}\label{s.1t}

In this section, we prove Theorem \ref{t.2} (1) on Frobenius eigenvalues. We
then extend the theorem of Drinfeld and Kedlaya on Newton polygons from
smooth schemes to normal stacks (Theorem \ref{t.slopes}). Finally, following
Drinfeld \cite[Appendix B]{Dr}, we study the category of weakly motivic
complexes, whose cohomology sheaves have ``motivic'' Frobenius eigenvalues
(Theorems \ref{t.mot} and \ref{t.D}).

\begin{notation}\label{n.W}
We let $M_0$ denote the group of algebraic numbers $\alpha\in \overline
\Q^\times$ of weight $0$ relative to $q$, namely, such that for every
Archimedean place $\lambda$ of $\Q(\alpha)$, we have $\lvert
\alpha\rvert_\lambda =1$. We let $W_0(q)\subseteq M_0$ denote the subgroup
of \emph{$q$-Weil numbers of weight $0$}.  For a subset $S\subseteq \Q$ of
slopes, we let $W^S_0(q)\subseteq W_0(q)$ denote the subgroup of $\alpha\in
W_0(q)$ of slopes in $S$, namely such that for every valuation $v$ on
$\Q(\alpha)$ satisfying $v(q)=1$, we have $v(\alpha)\in S$.
\end{notation}

Note that $W_0(q)$ only depends on the characteristic $p$ of $\F_q$ and that
$W^{\{0\}}_0(q)$ is simply the set of roots of unity in $\Qb$.

In the notation above, Theorem \ref{t.2} (1) says that the eigenvalues of
$\Fr_x$ acting on $\cF_{\bar x}$ belong to
$W^{[-\frac{r-1}{2},\frac{r-1}{2}]\cap \Q}_0(q^n)$ for all $x\in
X(\F_{q^n})$ and all $n\ge 1$. As mentioned earlier, we prove this by
reducing to the case of schemes. For the reduction to work, we need to show
that the statement can be checked on any dense open substack. We start by
reviewing Deligne's argument of joining by curves \cite[Proposition
1.9]{Deligne} and extending it to stacks.

\begin{prop}\label{p.eigen}
Let $X$ be a connected stack. Then there exists an integer $M\ge 1$ such
that, for every lisse Weil $\Qlb$-sheaf $\cF$ on $X$ of rank $r\ge 1$, and
for all $m,n\ge 1$, $x\in X(\F_{q^m})$, $y\in X(\F_{q^n})$, if we let
$\alpha_1,\dots,\alpha_r$ (resp.\ $\beta_1,\dots,\beta_r$) denote the
eigenvalues of $\Fr_x$ (resp.\ $\Fr_y$) acting on $\cF_{\bar x}$ (resp.\
$\cF_{\bar y}$), then, up to reordering, we have
$\beta_i^{1/n}/\alpha_i^{1/m}\in W_0^{[-M(r-1),M(r-1)]\cap \Q}(q)$ for $1\le
i\le r$.
\end{prop}

In the situation of the proposition, if $\alpha_1,\dots,\alpha_r\in W_0(q)$,
then $\beta_1,\dots,\beta_r\in W_0(q)$.

\begin{proof}
Let $f\colon Y\to X$ be a submersive morphism (for example, a flat
presentation) with $Y$ a separated algebraic space. Consider the
intersection graph $\Gamma$ of the irreducible components of $Y$ over $X$.
The vertices are the irreducible components of $Y$. There is an edge between
two vertices $v$ and $w$ if and only if the corresponding components $Y_v$
and $Y_w$ are such that $Y_v\times_X Y_w$ is nonempty. If $Y_v$ and $Y_w$
are on the same connected component of $Y$, then $v$ and $w$ belong to the
same connected component of~$\Gamma$. Thus for each component $V$ of
$\Gamma$, $Y_V\colonequals \bigcup_{v\in V}Y_v$ is a union of connected
components of $Y$. It follows that $f(Y_V)$ is open and closed, because
$f^{-1}(f(Y_V))=Y_V$. Since $X$ is connected, so is $\Gamma$.

We take $M-1$ to be the diameter of the graph $\Gamma$. For each edge
$e=(v,w)$ of the graph, choose a closed point $x_e$ of $Y_v\times_X Y_w$.
For $x$ and $y$ as in the statement of the proposition, let $v$ and $w$ be
vertices such that $\bar x$ lifts to $Y_v$ and $\bar y$ lifts to $Y_w$. Let
$v=v_1\frac{e_1}{} v_2 \dotsm v_{N-1} \frac{e_{N-1}}{} v_N=w$ be a path of
length $N-1\le M-1$. By Lemma \ref{l.curve} below, there exists a diagram
above $X$
\[x_0\to C_1 \leftarrow x_1 \to \dots \leftarrow x_{N-1} \to C_N \leftarrow x_N, \]
where $C_j$, $1\le j\le N$ are irreducible smooth curves over $\F_q$ above
$Y_{v_j}$, and $x_j=\Spec(\F_{q^{n_j}})$, $0\le j\le N$ such that $x_0$ is
above $x$, $x_j$ is above $x_{e_j}$ for $1\le j\le N-1$,  and $x_N$ is
above~$y$. We apply the proof of L.~Lafforgue's theorem \cite[Th\'eor\`eme
VII.6]{Lafforgue} (or V.~Lafforgue's improvement of the bound
\cite[Corollaire 2.2]{VLaff}) to $C_j$, $1\le j\le N$, and to the simple
factors of the pullback of $\cF$ to $C_j$. If $\alpha_1^{(j)},\dots,
\alpha_r^{(j)}$ denote the eigenvalues of $\Fr_{x_j}$, then up to reordering
the $r$ values, we have
$(\alpha_i^{(j)})^{1/n_j}/(\alpha_i^{(j-1)})^{1/n_{j-1}}\in
W_0^{[-(r-1),(r-1)]}(q)$.
\end{proof}

\begin{lemma}\label{l.curve}
Let $Y$ be an irreducible algebraic space separated of finite type over a
field~$k$ and let $x$ and $y$ be closed points of $Y$. There exists an
irreducible regular curve $C$ and a morphism $C\to Y$ of image containing
$x$ and $y$.
\end{lemma}

\begin{proof}
By Chow's lemma \cite[IV Theorem 3.1]{Knutson}, we may assume that $Y$ is a
scheme. Replacing $Y$ by an irreducible component of $Y\otimes_k \bar k$, we
may assume that $k$ is algebraically closed. For this case, see
\cite[Section~6]{Mumford}.
\end{proof}

\begin{prop}\label{p.int}
Let $R$ be an integrally closed subring of $\Qlb$. Let $j\colon U\to X$ be a
dominant open immersion of stacks and let $\cF$ be a lisse Weil $\Qlb$-sheaf
on $X$ such that $j^*\cF$ is $R$-integral (resp.\ inverse $R$-integral).
Then $\cF$ is $R$-integral (resp.\ inverse $R$-integral).
\end{prop}

Recall that, following \cite[Variante 5.13, D\'efinition 6.1]{Zint}, a Weil
$\Qlb$-sheaf $\cF$ on $X$ is said to be \emph{$R$-integral} (resp.\
\emph{inverse $R$-integral}) if for all $n\ge 1$ and all $x\in X(\F_{q^n})$,
the eigenvalues (resp.\ inverse eigenvalues) of $\Fr_x$ acting on $\cF_{\bar
x}$ are in $R$.

\begin{proof}
Up to replacing $X$ by a smooth presentation, we may assume that $X$ is a
scheme. Up to replacing $X$ by its normalization, we may further assume that
$X$ is normal. In this case, $\cF\simeq j_*j^*\cF$ is $R$-integral (resp.\
inverse $R$-integral) by \cite[Th\'eor\`eme 2.2, Variantes 5.1, 5.13]{Zint}.
(The integral case is a theorem of Deligne \cite[XXI Th\'eor\`eme
5.6]{SGA7II} assuming resolution of singularities.)
\end{proof}

\begin{remark}\label{r.shrink}
Let $I=[a,b]\cap \Q$ be an interval with $a,b\in \Q$. It follows from the
propositions that if $\cF$ is a lisse $\Qlb$-sheaf on a stack $X$ such that
the eigenvalues of $\Fr_x$ acting on $\cF_{\bar x}$ belong to $W^I_0(q^n)$
for all $n\ge 1$ and all $x\in U(\F_{q^n})$, where $U$ is some dense open
substack of $X$, then the same holds for all $x\in X(\F_{q^n})$. Indeed, the
eigenvalues belong to $W_0(q)$ by Proposition \ref{p.eigen}, and
$\cF^{(q^{-a})}$ is $\Zb$-integral and $\cF^{(q^{-b})}$ is inverse
$\Zb$-integral (for all representatives of $q^{-a}$ and $q^{-b}$) by
Proposition \ref{p.int}. Here $\Zb$ denotes the ring of algebraic integers.
\end{remark}

\begin{proof}[Proof of Theorem \ref{t.2} (1)]
By Remark \ref{r.shrink}, we may shrink $X$. Thus, by Lemma \ref{l.Behrend},
we are reduced to the case of Deligne-Mumford stacks. Up to shrinking $X$,
we may assume that there exists a finite \'etale cover $f\colon Y\to X$,
where $Y$ is a scheme. By Lemma \ref{l.finite}, $f^*\cF\simeq \bigoplus_i
\cF_i$, with $\cF_i$ simple and $\det(\cF_i)$ of finite order. We are thus
reduced to the case where $X$ is a scheme. This case was stated in
\cite[Proposition VII.7]{Lafforgue}, and the gap in the proof has been fixed
by Deligne \cite[Th\'eor\`eme 1.6]{Deligne} and others. Indeed, by Remark
\ref{r.shrink} again, we may assume that $X$ is a smooth separated scheme.
By a consequence of Hilbert irreducibility (\cite[Proposition 2.17]{Dr} or
\cite[Proposition B.1]{EK}), for any closed point $x$ of $X$, there exists a
smooth curve $C$ over $\F_q$ and a morphism $g\colon C\to X$ such that $x$
is in the image of $g$ and $g^*\cF$ is simple. It then suffices to apply
L.~Lafforgue's theorem for curves \cite[Th\'eor\`eme VII.6]{Lafforgue} and
V.~Lafforgue's improvement of the bound \cite[Corollaire 2.2]{VLaff}.
\end{proof}

More generally V.~Lafforgue proved an inequality for the Newton polygon in
the case of curves. Recently Drinfeld and Kedlaya \cite[Theorem 1.3.3]{DK}
gave a refinement for the lowest Newton polygon in the case of smooth
schemes. These results extend to normal stacks as follows.

For a stack $X$, we let $\lvert X(\Fqb)\rvert$ denote the set of isomorphism
classes of the groupoid $X(\Fqb)$. We let $\lvert X\rvert$ denote the set of
orbits of $\Gal(\Fqb/\F_q)$ acting on $\lvert X(\Fqb)\rvert$. If $X$ is a
Deligne-Mumford stack, then $\lvert X\rvert$ can be identified with the set
of closed points of $X$. In general, following \cite[Lemma 5.3.4]{DK}, we
equip $\lvert X\rvert$ with the following topology $T$: a subset $U\subseteq
\lvert X\rvert$ is $T$-open if and only if for every morphism $C\to X$ from
a smooth curve $C$ to $X$, the inverse image of $U$ under the map $\lvert
C\rvert \to \lvert X\rvert$ is open for the Zariski topology on $\lvert C
\rvert$.

We say that a Weil $\Qlb$-sheaf $\cF$ on $X$ is \emph{algebraic} if it is
$\Qb$-integral. We fix a valuation $v$ on $\overline{\Q}$ such that
$v(q)=1$. For an algebraic Weil $\Qlb$-sheaf $\cF$ on $X$ and $x\in
X(\F_{q^n})$, we let $s_1^x(\cF)\le \dots \le s_r^x(\cF)$ denote the images
under $v/n$ of the eigenvalues of $\Fr_x$ acting on $\cF_{\bar x}$. These
rational numbers are called the \emph{slopes} of $\cF$ at $x$ and depend on
$x$ only through the image of $x$ in $\lvert X\rvert$.

\begin{theorem}\label{t.slopes}
Let $X$ be an irreducible stack. Let $\cF$ be an algebraic lisse Weil
$\Qlb$-sheaf of rank $r$ on $X$. Then
\begin{enumerate}
\item There exist rational numbers $s_1(\cF)\le \dots \le s_r(\cF)$ such
    that $\sum_{j=1}^i s_j(\cF)\le \sum_{j=1}^i s_j^x(\cF)$ for all $x$
    and all $i$ and the set $Y\subseteq \lvert X\rvert$ of $y$ satisfying
    $s_i^y(\cF)=s_i(\cF)$ for all $i$ is nonempty and $T$-open.
\item If $X$ is geometrically unibranch and $\cF$ is indecomposable, then
    $s_{i+1}(\cF)\le s_i(\cF)+1$ for all $1\le i\le r-1$.
\end{enumerate}
\end{theorem}

The numbers $s_i(\cF)$ are the slopes of the lowest Newton polygon.

\begin{remark}
In (2), if moreover $\det(\cF)$ has finite order, so that $\sum_{j=1}^r
s_j^x(\cF)=0$, then, as in \cite[Proof of Corollary 1.1.7]{DK}, the theorem
implies
\[\sum_{j=1}^i
s_j^x(\cF)\ge \sum_{j=1}^i
s_j(\cF)\ge -i(r-i)/2
\]
for all $x$ and all $i$. Taking $i=1$ and $i=r-1$, we recover the bounds
$s_1^x(\cF)\ge -(r-1)/2$ and $s_{r}^x(\cF)\le (r-1)/2$ in Theorem \ref{t.2}
(1).
\end{remark}

To prove the theorem, we need a couple of lemmas, extending \cite[Lemmas
5.3.1, 5.3.3, 5.3.4]{DK}.

\begin{lemma}\label{l.semicon}
Let $X$ be a stack and let $\cF$ be an algebraic lisse Weil $\Qlb$-sheaf on
$X$. For all $i$, the function $x\mapsto \sum_{j=1}^i s_j^x(\cF)$ on $\lvert
X\rvert$ is upper semi-continuous for the topology $T$, bounded, and takes
values in $N^{-1}\Z$ for some $N$.
\end{lemma}

\begin{proof}
By the definition of the topology $T$, for the semi-continuity we may assume
that $X$ is a smooth curve. We reduce then to the case $\cF$ simple, and
then to the case where $\det(\cF)$ has finite order. In this case, the
semi-continuity follows from Abe's theorem on crystalline companions
\cite[Theorem 4.4.1]{Abe} and the corresponding statement for overconvergent
$F$-isocrystals. The boundedness follows from Proposition \ref{p.eigen}. The
last assertion follows from the fact that $E(\cF)$ is a number field
(Theorem \ref{t.gen}) by the proof of \cite[Lemma 5.3.1]{DK}.
\end{proof}

\begin{lemma}\label{l.irr}
Let $X$ be an irreducible stack. Then $\lvert X\rvert$ is irreducible for
the topology $T$.
\end{lemma}

\begin{proof}
If $f\colon Y\to X$ is a surjective morphism of stacks with $Y$ irreducible,
then $\lvert f \rvert$ is a surjection and we may replace $X$ by $Y$. Thus,
replacing $X$ by its normalization, we may assume $X$ is normal. Next note
that if $X$ admits a Zariski open cover $(X_i)$ such that $\lvert X_i\rvert$
is $T$-irreducible for all $i$, then $X$ is $T$-irreducible. Indeed, $\lvert
X_i\rvert \cap \lvert X_j\rvert \neq \emptyset$ and \cite[0 2.1.4]{EGAI}
applies. Let $f\colon Y \to X$ be a flat presentation with $Y$ a separated
scheme. Each connected component $Y_i$ of $Y$ is irreducible, and $(f(Y_i))$
is a Zariski open cover of $X$. We may thus replace $X$ by $Y_i$, and assume
that $X$ is a separated scheme. Let $U_1,U_2\subseteq \lvert X\rvert$ be
nonempty $T$-open subsets. Let $x_i\in U_i$. By Lemma \ref{l.curve}, there
exists a morphism $g\colon C\to X$, where $C$ is an irreducible smooth
curve, such that $x_1$ and $x_2$ are in the image of $\lvert g\rvert$. Then
$\lvert g\rvert^{-1}(U_i)$ is nonempty for $i=1,2$. It follows that $\lvert
g\rvert^{-1}(U_1\cap U_2)$ and hence $U_1\cap U_2$ are nonempty.
\end{proof}

\begin{proof}[Proof of Theorem \ref{t.slopes}]
(1) By Lemma \ref{l.semicon}, the function $a_i\colon x\mapsto
    \sum_{j=1}^i
    s_j^x(\cF)$ on $\lvert X\rvert$ attains a minimum. We define $s_1\le \dots \le s_r$ so that the minimum of $a_i$ is $\sum_{j=1}^i {s_j}$. Moreover, the locus $Y_i\subseteq \lvert X\rvert$
    on which $a_i$ attains the minimum is $T$-open. Therefore,
    $Y=\bigcap_{i=1}^{r-1}
    Y_i$ is nonempty and $T$-open by Lemma \ref{l.irr}.

(2) Since $\lvert X\rvert$ is irreducible, we may shrink $X$. Thus, by Lemma
\ref{l.Behrend}, we may assume that $X$ is a Deligne-Mumford stack. Further
shrinking $X$, we may assume that $X=[Y/G]$ is the quotient stack of a
smooth affine scheme $Y$ by a finite group $G$. Choose an embedding $G\to
\GL_m$. Then $Y\wedge^G \GL_m =(Y\times \GL_m)/G$ is a $\GL_m$-torsor over
$[Y/G]$. By Lemma \ref{l.subm}, we may replace $X$ by the smooth affine
scheme $Y\wedge^G \GL_m$. This case is \cite[Theorem 1.3.3]{DK}.
\end{proof}

Let $\iota\colon \Qlb\to \C$ be an embedding.  Following \cite[2.4.3]{SunL},
we say that a Weil $\Qlb$-sheaf $\cF$ on a stack $X$ is \emph{punctually
$\iota$-pure of weight $w\in \R$} if for every $x\in X(\F_{q^n})$, $n\ge 1$
and every eigenvalue $\alpha$ of $\Fr_x$ acting on $\cF_{\bar x}$, we have
$\lvert \iota\alpha\rvert =q^{wn/2}$. The results of Sun in \cite{SunL} and
\cite{Sun} extend to stacks not necessarily of separated diagonal by Lemma
\ref{l.sep} and to Weil $\Qlb$-sheaves by Proposition \ref{p.Weil}. For
$w\in \Z$, we say that $\cF$ is \emph{punctually pure of weight $w$} if it
is punctually $\iota$-pure of weight $w$ for all $\iota$.

\begin{remark}\label{r.mixed}
By Theorem \ref{t.2} (1) and Proposition \ref{p.twist}, every simple lisse
Weil $\Qlb$-sheaf on a geometrically unibranch stack is punctually
$\iota$-pure. It follows that every Weil $\Qlb$-sheaf on a stack is
$\iota$-mixed, namely, a successive extension of punctually $\iota$-pure
sheaves (cf.\ \cite[Remark 2.8.1]{SunL}). Similarly, a Weil $\Qlb$-sheaf
$\cF$ on a stack is mixed, namely, a successive extension of punctually pure
sheaves (of integral weights), if and only if for every $x\in X(\F_{q^n})$,
$n\ge 1$, the eigenvalues of $\Fr_x$ acting on $\cF_{\bar x}$ belong to
$M(q^n)$. Here $M(q)\colonequals \bigcup_{w\in \Z} q^{w/2}M_0$. (Recall that
$M_0$ is the group of algebraic numbers of weight $0$.)
\end{remark}

The structure of punctually $\iota$-pure Weil $\Qlb$-sheaves can be
described as follows. We let $\cE_n$ denote the $\Qlb$-sheaf on
$\Spec(\F_q)$ of stalk $\Qlb^n$ on which $\Fr_q$ acts unipotently with one
Jordan block.

\begin{prop}\label{p.purelisse}
Let $X$ be a geometrically unibranch stack. Then indecomposable punctually
$\iota$-pure lisse Weil $\Qlb$-sheaves are of the form $\cF\otimes
\pi_X^*\cE_n$ with $\cF$ simple, where $\pi_X\colon X\to \Spec(\F_q)$.
\end{prop}

In the appendix we will prove an analogue for pure perverse sheaves. The
proposition still holds with $\Qlb$ replaced by a finite (or algebraic)
extension of $\Q_\ell$.

\begin{proof}
As in the case of perverse sheaves on schemes \cite[Proposition 5.3.9
(i)]{BBD}, this follows from the geometric semisimplicity of punctually
$\iota$-pure lisse Weil $\Qlb$-sheaves \cite[Theorem 2.1 (iii)]{Sun}
(generalizing \cite[Th\'eor\`eme 3.4.1 (iii)]{WeilII}).
\end{proof}

Let $W(q)=\bigcup_{w\in \Z} q^{w/2}W_0(q)$ be the group of $q$-Weil numbers
(of integral weights). We say that $K\in D(X,\Qlb)$ is \emph{weakly motivic}
if for all $n\ge 1$, $x\in X(\F_{q^n})$, and $i\in\Z$, the eigenvalues of
$\Fr_x$ acting on $\cH^i K_{\bar{x}}$ belong to $W(q^n)$. We let
$D_{\mot}(X,\Qlb)\subseteq D(X,\Qlb)$ denote the full subcategory spanned by
weakly motivic complexes, which is a thick subcategory. For $*\in
\{+,-,b\}$, we put $D^*_\mot=D^*\cap D_\mot$.

\begin{remark}
By Proposition \ref{p.eigen}, for a lisse $\Qlb$-sheaf $\cF$ on a connected
stack $X$ and a fixed $x\in X(\F_{q^{n}})$, $\cF$ is weakly motivic if and
only if the eigenvalues of $\Fr_{x}$ acting on $\cF_{\bar {x}}$ are in
$W(q^{n})$.
\end{remark}

The following result generalizes \cite[Theorems B.3, B.4]{Dr}.

\begin{theorem}\label{t.mot}
Let $f$ be a morphism of stacks. The six operations and Grothendieck-Verdier
duality induce
\begin{gather*}
\otimes \colon D^-_{\mot}(X,\Qlb)\times D^-_{\mot}(X,\Qlb)\to D^-_{\mot}(X,\Qlb),\\
R\cHom\colon D^-_{\mot}(X,\Qlb)^{\op}\times D^+_{\mot}(X,\Qlb)\to D^+_{\mot}(X,\Qlb),\\
D\colon D_{\mot}(X,\Qlb)^{\op}\to D_{\mot}(X,\Qlb),\quad f^*, f^!\colon D_{\mot}(Y,\Qlb)\to D_{\mot}(X,\Qlb),\\
f_*\colon D^+_{\mot}(X,\Qlb)\to D^+_{\mot}(Y,\Qlb),\quad f_!\colon D^-_{\mot}(X,\Qlb)\to D^-_{\mot}(Y,\Qlb).
\end{gather*}
If $f$ is relatively Deligne-Mumford, then we also have
\[f_*\colon D_{\mot}(X,\Qlb)\to D_{\mot}(Y,\Qlb),\quad f_!\colon D_{\mot}(X,\Qlb)\to D_{\mot}(Y,\Qlb).
\]
\end{theorem}

\begin{proof}
Note that $W(q)=R(q)^\times\cap M(q)$, where $R(q)$ is the integral closure
of $\Z[1/q]$ in $\overline{\Q}$. By \cite[Variante 5.13, Section 6]{Zint}
(which extends easily to stacks not necessarily of separated diagonals),
complexes with $R$-integral (resp.\ inverse $R$-integral) cohomology sheaves
are preserved by the six operations and duality. By Remark \ref{r.mixed},
having Frobenius eigenvalues in $M(q^n)$ is equivalent to being mixed, and
complexes with mixed cohomology sheaves are preserved by the operations by
\cite[Remark 2.12]{SunL}.
\end{proof}

As in \cite[Stabilit\'es 5.1.7]{BBD}, the theorem has the following
consequence.

\begin{cor}
The perverse truncation functors on $D(X,\Qlb)$ preserve $D_{\mot}(X,\Qlb)$
and induce a $t$-structure on $D_{\mot}(X,\Qlb)$.
\end{cor}

Theorems \ref{t.2} (1) and \ref{t.mot} imply the following result on the
structure of $D^b(X,\Qlb)$. For $a\in \Zlb^\times$, we let
$D^b_\mot(X,\Qlb)^{(a)}\subseteq D^b(X,\Qlb)$ denote the full subcategory
spanned by objects of the form $K^{(a)}$ with $K\in D^b_\mot(X,\Qlb)$. By
definition, $D^b_\mot(X,\Qlb)^{(a)}$ only depends on the class of $a$ in
$\Zlb^\times/W(q)$.

\begin{theorem}\label{t.D}
For any stack $X$, we have a canonical decomposition for the bounded derived
category of $\Qlb$-sheaves:
\[D^{b}(X,\Qlb)\simeq \bigoplus_{a\in
\Zlb^\times /W(q)} D^b_\mot(X,\Qlb)^{(a)}.
\]
\end{theorem}

The case of schemes is \cite[Theorem B.7]{Dr}.

\begin{proof}
The proof is very similar to the case of schemes and parallel to the proof
of Proposition \ref{p.Weil}. It suffices to show the following:
\begin{itemize}
\item (generation) Every object of $D^b(X,\Qlb)$ is a successive extension
    of objects of $D^b_\mot(X,\Qlb)^{(a)}$;
\item (orthogonality) $\Hom(A^{(a)},B^{(b)})=0$ for $A,B\in
    D^b_\mot(X,\Qlb)$, $a/b\not\in W(q)$.
\end{itemize}
The first point follows from Proposition \ref{p.twist} and Theorem \ref{t.2}
(1). For the orthogonality, note that
\[\Hom(A^{(a)},B^{(b)})\simeq
H^0(\Spec(\F_q),R\pi_{X*} R\cHom(A,B)^{(b/a)})=0,
\]
where $\pi_X\colon X\to \Spec(\F_q)$. Here we used the fact that $R\pi_{X*}
R\cHom(A,B)$ is in $D^+_\mot(\Spec(\F_q),\Qlb)$ by Theorem \ref{t.mot}.
\end{proof}

\begin{remark}\label{r.Pervmot}
The same decomposition holds for categories of $\Qlb$-sheaves and perverse
$\Qlb$-sheaves. In particular, the subcategory of weakly motivic perverse
sheaves $\Perv_\mot(X,\Qlb)\subseteq \Perv(X,\Qlb)$ is stable under
subquotient.
\end{remark}

\section{Frobenius traces}\label{s.2}

Theorem \ref{t.2} (2) follows immediately from Theorem \ref{t.2} (1) and the
following.

\begin{theorem}\label{t.gen}
Let $X$ be a stack and let $\cF$ be a Weil $\Qlb$-sheaf on $X$. Then
$E(\cF)$ is a finitely generated extension of $\Q$. In particular, $E(\cF)$
is a number field if and only if for all $n\ge 1$ and all $x\in
X(\F_{q^n})$, the eigenvalues of $\Fr_x$ acting on $\cF_{\bar x}$ are
algebraic numbers.
\end{theorem}

The case of schemes is a theorem of Deligne \cite[Th\'eor\`eme 3.1, Remarque
3.9]{Deligne}.

\begin{proof}
To show the first assertion, by induction, we may replace $X$ by a dense
open substack. In particular, we may assume that $\cF$ is lisse. Moreover,
by Lemma \ref{l.Behrend} and Remark \ref{r.gerbe} (or the fact that any
sub-extension of a finitely generated field extension is finitely generated
\cite[page V.113, Corollaire 3]{BAlg}), we may assume that $X$ is a
Deligne-Mumford stack. We may further assume that $X\simeq [Y/G]$ for a
finite group $G$ acting on an affine scheme $Y$. Choose an embedding $G\to
\GL_m$ and take $Z=Y\wedge^G \GL_m=(Y\times\GL_m)/G$. Then $f\colon Z\to X$
is a $\GL_m$-torsor. We have $E(f^*\cF)=E(\cF)$. Indeed, any point $x\in
X(\F_{q^n})$ lifts to a point of $Z(\F_{q^n})$ by Hilbert's Theorem 90. We
then apply the case of schemes \cite[Th\'eor\`eme 3.1, Remarque
3.9]{Deligne} to $f^*\cF$ on $Z$. For the second assertion, it suffices to
note that the Frobenius eigenvalues are algebraic numbers if and only if the
Frobenius traces are algebraic numbers.
\end{proof}

\begin{cor}\label{c.Esub}
Let $\cF$ be a Weil $\Qlb$-sheaf on a stack $X$ and let $\cG$ be a
subquotient of $\cF$. Then $E(\cG)$ is contained in a finite extension of
$E(\cF)$.
\end{cor}

\begin{proof}
For each $x\in X(\F_{q^n})$, the eigenvalues of $\Fr_{x}$ on $\cF_{\bar x}$
are contained in a finite extension of $E(x^*\cF)\subseteq E(\cF)$. The
assertion then follows from the theorem, which says that $E(\cG)$ is
generated by the traces $\tr(\Fr_{x},\cG_{\overline{x}})$ at a finite number
of points $x$ with varying $n$.
\end{proof}

For any morphism $f\colon X\to Y$ of stacks and any Weil $\Qlb$-sheaf $\cG$
on $Y$, we have $E(f^*\cG)\subseteq E(\cG)$.

\begin{cor}\label{c.finite2}
Let $f\colon X\to Y$ be a morphism  of stacks with $X$ nonempty and $Y$
connected. For any \emph{lisse} Weil $\Qlb$-sheaf $\cG$ on $Y$, the field
$E(\cG)$ is a finite extension of $E(f^*\cG)$. In particular, for any $n\ge
1$ and any $y\in Y(\F_{q^{n}})$, the field $E(\cG)$ is a finite extension of
$E(y^*\cG)$ (the field generated by $\tr(\Fr^m_{y},\cF_{\bar y})$, $m\ge
1$).
\end{cor}

\begin{proof}
For any $y'\in X(\F_{q^{n'}})$, the eigenvalues of $\Fr_{y'}$ on
$\cG_{\overline{y'}}$ are contained in a finite extension of $E(y^*\cG)$ by
Proposition \ref{p.eigen}. The second assertion then follows from the
theorem, which says that $E(\cG)$ is generated by the traces
$\tr(\Fr_{y'},\cG_{\overline{y'}})$ at a finite number of points $y'$ with
varying $n'$. For the first assertion, let $x\in X(\F_{q^n})$ and let
$y=f(x)$. Then $E(y^*\cG)\subseteq E(f^*\cG)\subseteq E(\cG)$ and the first
assertion follows from the second one.
\end{proof}

\begin{cor}\label{c.finite}
Let $f\colon X\to Y$ be a surjective morphism of stacks. Then, for any Weil
$\Qlb$-sheaf $\cG$ on $Y$, the field $E(\cG)$ is a finite extension of
$E(f^*\cG)$.
\end{cor}

\begin{proof}
This follows from Corollary \ref{c.finite2} by taking a stratification of
$Y$ by connected strata such that the restriction of $\cG$ to each stratum
is lisse.
\end{proof}

Under additional assumptions, Corollary \ref{c.finite2} admits the following
refinement, which is a consequence of Gabber's theorem on the preservation
of companionship.

\begin{prop}\label{p.Eopen}
Let $f\colon X\to Y$ be a dominant open immersion of smooth stacks, $Y$
having separated diagonal, and let $\cG$ be a lisse Weil $\Qlb$-sheaf on
$Y$. Then $E(f^*\cG)=E(\cG)$. Moreover, if $\cG'$ is a lisse Weil
$\Qlbp$-sheaf on $Y$ such that $f^*\cG'$ is a $\sigma$-companion of $f^*\cG$
for some embedding $\sigma\colon E(\cG)\to \Qlbp$, then $\cG'$ is a
$\sigma$-companion of $\cG$.
\end{prop}

\begin{proof}
By the existence of smooth neighborhood \cite[Th\'eor\`eme 6.3]{LMB} (here
we used the assumption that $Y$ has separated diagonal), any point $y\in
Y(\F_{q^n})$ factorizes through a smooth morphism $Y'\to Y$, where $Y'$ is a
scheme. We are thus reduced to the case of schemes, which is a special case
of \cite[Proposition 3.10]{Zind} (case $K$ a finite field and $G=\{1\}$),
consequence of Gabber's theorem and purity.
\end{proof}

As pointed out to us by Drinfeld, another way to prove the case of schemes
of the proposition is to reduce by Drinfeld's version of Hilbert
irreducibility \cite[Theorem 2.15]{Dr} to the case of smooth curves, which
is a theorem of Deligne \cite[Th\'eor\`eme 9.8]{Delignec}.

\section{Companions}\label{s.3}
In this section, we prove Theorem \ref{t.1} on the existence of lisse
companions on smooth stacks of separated diagonal and Corollary
\ref{c.moduli} on the existence of lisse companions on coarse moduli spaces.
We then deduce the existence of perverse companions on stacks of separated
diagonal (Theorem \ref{t.perverse}). We also deduce that companionship
induces isomorphisms among the Grothendieck groups of Weil $\Qlb$-sheaves
for varying $\ell$ (Corollary \ref{c.K}).

To apply the reduction steps to Theorem \ref{t.1}, again we need to show
that we may shrink $X$. This is done by combining Proposition \ref{p.Eopen}
with Drinfeld's theorem on the existence of companions on schemes. Let
$E_{\lambda'}$ be an algebraic extension of $\Q_{\ell'}$.

\begin{prop}\label{p.comp}
Let $j\colon U\to X$ be a dominant open immersion of smooth stacks, $X$
having separated diagonal. Let $\cF$ be a lisse Weil $\Qlb$-sheaf on $X$ and
let $\sigma\colon E(\cF) \to E_{\lambda'}$ and $\iota'\colon E_{\lambda'}\to
\C$ be  embeddings. Assume that $j^*\cF$ admits a lisse punctually
$\iota'$-pure $\sigma$-companion $\cG'$. Then $j_*\cG'$ is a lisse
$\sigma$-companion of $\cF$.
\end{prop}

\begin{proof}
If $j_*\cG'$ is lisse, then $j_*\cG'$ is a $\sigma$-companion of $\cF$ by
Proposition \ref{p.Eopen}. It remains to show that $j_*\cG'$ is lisse. For
this we may assume $E_{\lambda'}=\overline{\Q_{\ell'}}$. Since any pullback
of $\cG'$ is punctually $\iota'$-pure, we may assume that $X$ is a scheme.
Since $\cG'$ is geometrically semisimple \cite[Th\'eo\`eme 3.4.1
(iii)]{WeilII}, $\cG'_{\Fqb}\simeq (\cG'^\rss)_{\Fqb}$, where $\cG'^{\rss}$
is the semisimplification of $\cG'$. Thus $j_*\cG'$ is lisse if and only if
$j_*\cG'^{\rss}$ is lisse. By Drinfeld's theorem, $\cF$ admits a lisse
$\sigma$-companion $\cF'$, which we may assume semisimple. By Chebotarev's
density theorem, we have $\cG'^{\rss}\simeq j^*\cF'$. Therefore, $j_*
\cG'^{\rss}\simeq j_*j^*\cF'\simeq \cF'$ is lisse.
\end{proof}

Assuming Theorem \ref{t.1} on the existence of lisse companions, we have the
following consequence of Chebotarev's density theorem.

\begin{prop}\label{p.simple}
Let $\cF$ be a lisse Weil $\Qlb$-sheaf on a geometrically unibranch stack
$X$ and let $\sigma\colon E(\cF)\to E_{\lambda'}$ be an embedding. Then
lisse $\sigma$-companions of $\cF$ are unique up to semisimplification.
Moreover, if $\cF$ is simple, then, up to isomorphism, there exists at most
one lisse $\sigma$-companion $\cF'$ of $\cF$, and $\cF'$ is simple if it
exists.
\end{prop}

It is convenient to slightly extend terminology as follows. Given a Weil
$\Qlb$-sheaf $\cF$ on a stack $X$ and an embedding $\sigma\colon E'\to
E_{\lambda'}$ where $E'$ is an extension of $E(\cF)$, we will refer to
$(\sigma\res E(\cF))$-companions of $\cF$ simply as $\sigma$-companions.

\begin{proof}
The first assertion follows from Chebotarev's density theorem (see
Proposition \ref{p.inj} below). For the second assertion, it suffices to
show the simplicity, as the uniqueness then follows from the first
assertion. Up to replacing $X$ by a dense substack, we may assume that $X$
is smooth and of separated diagonal (Lemma \ref{l.sep}). We may further
assume that $E_{\lambda'}=\overline{\Q_{\ell'}}$. Extend $\sigma$ to an
isomorphism $\Qlb\simto \Qlbp$, which we still denote by $\sigma$. Let
$\cF'$ be a lisse $\sigma$-companion of $\cF$. Up to replacing $\cF'$ by its
semisimplification, we may assume $\cF'\simeq \bigoplus_i \cF'_i$, with each
$\cF'_i$ simple. By Theorem \ref{t.1}, there exists a
$\sigma^{-1}$-companion (i.e.\ $(\sigma^{-1}\res E(\cF'_i))$-companion)
$\cF_i$ of $\cF'_i$. Then $\cF$ is the semisimplification of
$\bigoplus_i\cF_i$, since they are both $\sigma^{-1}$-companions of $\cF'$.
Thus $\cF_i=0$ for all but one $i$, and the same holds for $\cF'_i$.
Therefore, $\cF'$ is simple.
\end{proof}

\begin{proof}[Proof of Theorem \ref{t.1}]
By Corollary \ref{c.Esub}, we may assume $X$ irreducible and $\cF$ simple.
In this case, we show in addition to the statements of the theorem, that any
lisse $\sigma$-companion $\cF'$ is punctually $\iota'$-pure for any
embedding $\iota'\colon \Qlbp\to \C$. By Proposition \ref{p.comp} and
Proposition \ref{p.Eopen} (or Corollary \ref{c.finite2}), we may shrink $X$.
Applying Lemma \ref{l.Behrend} and Remark \ref{r.gerbe} (or Corollary
\ref{c.finite}), we reduce to the case where $X$ is a Deligne-Mumford stack.
Up to shrinking $X$, we may further assume that $X=[Y/G]$, where $G$ is a
finite group acting on an affine scheme $Y$. Choose an embedding $G\to
\GL_m$. Consider the embedding $\GL_m\to \A^{m^2}$, which is equivariant
under the action of $\GL_m$. Let $p\colon Z=[Y\times \A^{m^2}/G]\to [Y/G]$
be the projection and let $s$ be the zero section. Note that $Z$ is smooth.
Since $s^*p^*\cF\simeq \cF$, $p^*\cF$ is simple. It suffices to show the
assertions for $(Z,p^*\cF)$. Indeed, if $\cG'$ is a lisse $\sigma$-companion
of $p^*\cF$, then $\cF'=s^*\cG'$ is a $\sigma$-companion of $s^*p^*\cF\simeq
\cF$. Applying Propositions \ref{p.comp} and \ref{p.Eopen} to the dense open
subscheme $[Y\times \GL_m/G]$ of $Z$, we are reduced to the case of schemes.
In this case, the existence of lisse companions follows from Drinfeld's
theorem (\cite[Theorem 1.1, Section 1.2]{Dr} applied to a twist $\cF^{(a)}$
of $\cF$ such that $\det(\cF^{(a)})$ has finite order). Moreover, if $\cF'$
is a lisse $\sigma$-companion of $\cF$, then $\cF'$ is simple by Proposition
\ref{p.simple} (which can be applied, because Proposition \ref{p.simple} for
schemes depends only on the known existence of lisse companions on smooth
schemes), hence punctually $\iota'$-pure.
\end{proof}

\begin{definition}
Let $f\colon X\to Y$ be a morphism of stacks. We say that $f$ \emph{creates
lisse companions} if for every lisse Weil $\Qlb$-sheaf $\cG$ on $Y$ and
every embedding $\sigma\colon E(\cG)\to E_{\lambda'}$ such that $f^*\cG$
admits a $\sigma$-companion, $\cG$ admits a $\sigma$-companion.
\end{definition}

Note that we do \emph{not} ask for the existence of a companion $\cG'$ of
$\cG$ such that $f^*\cG'$ is isomorphic to a given companion of $f^*\cG$.
Morphisms creating lisse companions are stable under composition. If
$X\xrightarrow{f} Y\xrightarrow{g} Z$ is a sequence of morphisms of stacks,
and if $gf$ creates lisse companions, then $g$ creates lisse companions.

\begin{prop}\label{p.create}
Let $f\colon X\to Y$ be a morphism of stacks. Then $f$ creates lisse
companions if it satisfies \emph{any} of the following conditions:
\begin{enumerate}
\item $f$ is a proper universal homeomorphism.
\item $f\colon X=\coprod_i X_i \to Y$, where $(X_i)$ is a finite Zariski
    open cover and $Y$ is geometrically unibranch.
\end{enumerate}
\end{prop}

\begin{proof}
Let $\cG$ be a lisse Weil $\Qlb$-sheaf on $Y$, let $\sigma\colon E(\cG)\to
E_{\lambda'}$ be an embedding, and let $\cF'$ be a lisse $\sigma$-companion
of $f^*\cG$.

(1) Let us first note that for any proper morphism $f$ of stacks with
geometrically connected fibers, the adjunction map $a \colon \cG\to
f_*f^*\cG$ is an isomorphism and the adjunction map $b\colon f^*f_*\cF'\to
\cF'$ is a monomorphism. Indeed, by proper base change, the stalk of $a$ at
any geometric point $\bar y\to Y$ can be identified with the isomorphism
$\cG_{\bar y}\to \Gamma(X_{\bar y},\cG_{\bar y})$, and the stalk of $b$ at a
geometric point $\bar x\to X$ above $\bar y\to Y$ can be identified with the
injection $\Gamma(X_{\bar y},\cF'|X_{\bar y'})\to \cF'_{\bar x}$. Now let
$f$ be a proper universal homeomorphism. By Lemma \ref{l.push} below,
$f_*\cF'$ is a $\sigma$-companion of $f_*f^*\cG\simeq \cG$. It remains to
show that $f_*\cF'$ is lisse. We have seen that $b$ is a monomorphism. Both
$f^*f_*\cF'$ and $\cF'$ are $\sigma$-companions of $f^*\cG$. In other words,
for any $x\in X(\F_{q^n})$, the characteristic polynomials of $\Fr_y$ acting
on $(f^*f_*\cF')_{\bar x}$ and $\cF'_{\bar x}$ coincide. In particular,
$(f^*f_*\cF')_{\bar x}$ and $\cF'_{\bar x}$ have the same rank. Thus, $b$ is
an isomorphism. By Lemma \ref{l.lisse}, $f_*\cF'$ is lisse.

(2) We may assume $Y$ irreducible and each $X_i$ nonempty. Then $U=\bigcap_i
X_i$ is nonempty. Let $j\colon U\to Y$ and let $\cG'_U$ be a semisimple
lisse $\sigma$-companion of $j^*\cG$. Let $\cF'_i= \cF'\res X_i$. Then
$\cG'_U\simeq \cF'^{\rss}_i\res U$, so that $j_*\cG'_U\res X_i\simeq
\cF'^{\rss}_i$. Thus $j_*\cG'_U$ is a lisse $\sigma$-companion of $\cG$.
\end{proof}

\begin{lemma}\label{l.push}
Let $f\colon X\to Y$ be a proper universal homeomorphism of stacks. Let
$\cF$ be a $\Qlb$-sheaf on $X$ and let $\cF'$ be a $\sigma$-companion of
$\cF$, where $\sigma\colon E(\cF)\to E_{\lambda'}$ is an embedding. Then
$E(f_*\cF)\subseteq E(\cF)$ and $f_*\cF'$ is a $\sigma$-companion of
$f_*\cF$.
\end{lemma}

\begin{proof}
By proper base change, we may assume that $Y=\Spec(\F_{q^n})$ is a point. We
may further assume that $X$ is reduced. In this case, $X=BG$ for a group
scheme $G$ over $Y$. Applying the proof of Lemma \ref{l.Behrend}, we get
$f=hg$, where $BG\xrightarrow{g} B(G/G^0)\xrightarrow{h} Y$. We have
$\cF\simeq g^*g_*\cF$ and $\cF'\simeq g^*g_*\cF'$. By Remark \ref{r.gerbe},
$E(g_*\cF)=E(\cF)$ and $g_*\cF'$ is a $\sigma$-companion of $g_*\cF$. We are
thus reduced to showing that $E(h_*-)\subseteq E(-)$ and that $h_*$
preserves $\sigma$-companions. For this case, we apply \cite[Proposition
5.8]{Zind} recalled as part of Theorem \ref{t.Gabber} below (or the trace
formula \cite[Theorem 4.2]{SunL}).
\end{proof}

\begin{proof}[Proof of Corollary \ref{c.moduli}]
By Proposition \ref{p.create} (2), we may assume that $X$ is the coarse
moduli space of a smooth stack $Y$ with finite inertia. It then suffices to
apply Proposition \ref{p.create} (1) to the proper universal homeomorphism
$f\colon Y\to X$ and Theorem \ref{t.1} to $Y$.
\end{proof}

In the rest of the section, we discuss companions of perverse sheaves and in
Grothendieck groups. For this, it is convenient to introduce perverse Weil
sheaves. Let $E_\lambda$ be an algebraic extension of $\Q_\ell$. A
\emph{perverse Weil $E_\lambda$-sheaf} on a stack $X$ is a perverse
$E_\lambda$-sheaf $\cP$ on $X\otimes_{\F_q} \Fqb$ equipped with an action of
the Weil group $W(\Fqb/\F_q)$ lifting the action of $W(\Fqb/\F_q)$ on
$X\otimes_{\F_q} \Fqb$. A morphism of perverse Weil $E_\lambda$-sheaves on
$X$ is a morphism of the underlying perverse $E_\lambda$-sheaves on
$X\otimes_{\F_q} \Fqb$ compatible with the action of $W(\Fqb/\F_q)$. As in
the case of schemes \cite[Proposition 5.1.2]{BBD} or $E_\lambda$-sheaves
(Remark \ref{r.ext}), we have a fully faithful functor
$\Perv^W(X,E_\lambda)\to \Perv(X,E_\lambda)$ and the essential image is
stable under extension. Remark \ref{r.Qlb} on extending scalars to $\Qlb$
still holds. The analogue of Proposition \ref{p.Weil} holds with the same
proof:
\[\Perv^W(X,\Qlb)\simeq\bigoplus_{a\in \Qlb^\times/\Zlb^\times} \Perv(X,\Qlb)^{(a)}.\]

We let $K^W_\lisse(X,E_\lambda)$ denote the Grothendieck group of
$\Sh_\lisse^W(X,E_\lambda)$, which is a free Abelian group generated by the
isomorphism classes of simple lisse Weil $E_\lambda$-sheaves on $X$. We let
$K^W(X,E_\lambda)$ denote the Grothendieck group of $\Sh^W(X,E_\lambda)$,
which is also the Grothendieck group of $\Perv^W(X,E_\lambda)$, and is a
free Abelian group generated by the isomorphism classes of simple perverse
Weil $E_\lambda$-sheaves on $X$. For a Weil sheaf or Weil perverse sheaf
$\cF$, we let $[\cF]$ denote its class in the Grothendieck group. We have a
commutative diagram
\[\xymatrix{K^W_\lisse(X,E_\lambda)\ar@{^{(}->}[d]\ar[r]& K^W(X,E_\lambda)\ar@{^{(}->}[d]\\
K^W_\lisse(X,\Qlb)\ar[r]^{s_X} & K^W(X,\Qlb)}
\]
of Abelian groups. That the vertical arrows are injections is standard (cf.\
\cite[page VIII.191, Th\'eor\`eme 1]{BAlg8}).

We have the following Chebotarev's density theorem, generalizing \cite[Lemma
4.1.4]{SZ}.

\begin{prop}\label{p.inj}
Let $X$ be a stack.
\begin{enumerate}
\item The homomorphism $t_X\colon K^W(X,\Qlb)\otimes_\Z \Qlb\to
    \Qlb^{\coprod_{n\ge 1} \lvert X(\F_{q^n})\rvert}$ sending $A$ to
    $t_X(A)\colon x\mapsto \tr(\Fr_x,A_{\bar x})$ is injective.
\item For $X$ irreducible and geometrically unibranch, the homomorphism
    $s_X\colon K^W_\lisse(X,\Qlb)\to K^W(X,\Qlb)$ is injective and the
    conjugates of the images of $\Fr_x$ in the fundamental group
    $\pi_1(X)$ form a dense subset.
\end{enumerate}
\end{prop}

Here as before $\lvert X(\F_{q^n})\rvert$ denotes the set of isomorphism
classes of the groupoid $X(\F_{q^n})$. The Frobenius traces are extended to
$K^W(X,\Qlb)\otimes_\Z \Qlb$ by linearity: for $A=\sum_i c_i [\cF_i]$ with
$c_i\in \Qlb$, $\tr(\Fr_x,A_{\bar x})=\sum_{i} c_i\tr(\Fr_x,(\cF_i)_{\bar
x})$.

\begin{proof}
We extend cohomological operations to $K^W(X,\Qlb)\otimes_\Z \Qlb$ by
linearity (cf.\ Remark \ref{r.sixop} below). Let $A\in \Ker(t_X)$. There
exists a stratification of $X$ by geometrically unibranch substacks
$j_\alpha\colon X_\alpha\to X$ such that for each $\alpha$, $j_\alpha^*A$
belongs to the image of $s_{X_\alpha}\otimes \Qlb$. We have
$t_{X_\alpha}(j_\alpha^* A)=0$ and $A=\sum_\alpha j_{\alpha !}j_\alpha^*A$
and for (1) it suffices to show $j_\alpha^*A=0$. Changing notation, it
suffices to show that for $X$ irreducible and geometrically unibranch,
$t_X(s_X\otimes \Qlb)$ is an injection. Note that this implies the
injectivity of $s_X$. (Alternatively we can apply the reduction in the proof
of \cite[Th\'eor\`eme 1.1.2]{LaumonTF}.)

For $x\in X(\F_{q^n})$ and $A\in K(X,\Qlb)^{(a)}$, $a\in \Qlb^\times$, the
reciprocal zeroes and roots of $\det(1-t\Fr_x,A_{\bar x})$ belong to
$a^n\Zlb^\times$. It follows that
\[t_X(K^W(X,\Qlb)\otimes_\Z \Qlb)=\bigoplus_{a\in \Qlb^\times/\Zlb^\times}t_X(K(X,\Qlb)^{(a)}\otimes_\Z
\Qlb).
\]
Thus it suffices to show that $t_X^\lisse\colonequals t_X(s_X\otimes \Qlb)$
is an injection on $K(X,\Qlb)\otimes_\Z \Qlb$. This is equivalent to the
density of the conjugates of $\Fr_x$ in the fundamental group $\pi_1(X)$
\cite[Corollary 27.13]{CR}.

Since for any nonempty open substack $U$ of $X$, $\pi_1(U)\to \pi_1(X)$ is a
surjection, we may shrink $X$. We reduce to the case of Deligne-Mumford
stacks as follows. Let $A=\sum_i c_i [\cF_i]$ be an element in the kernel of
$t_X^\lisse$ with $\cF_i\in \Sh_\lisse(X,\Qlb)$ and $c_i\in \Qlb$. We apply
Lemma \ref{l.Behrend} to $\bigoplus_i \cF_i$ to find, up to shrinking $X$, a
gerbe-like morphism $f\colon X\to Y$ where $Y$ is a Deligne-Mumford stack
such that $f^*f_*\cF_i\simeq \cF_i$ for all $i$. By Remark \ref{r.gerbe},
$t_Y^\lisse(f_* A)=0$, where $f_*A=\sum_{i} c_i [f_*\cF_i]$. Thus, by
Chebotarev's density theorem for Deligne-Mumford stacks \cite[Lemmas 4.1.4,
4.1.5]{SZ}, we have $f_* A=0$. Therefore, $A=f^*f_* A=0$.
\end{proof}

The definitions of $E(\cF)$ and $\sigma$-companions at the beginning of the
paper extend to elements of Grothendieck groups and to perverse Weil sheaves
as follows. Given $A\in K^W(X,\Qlb)$, we let $E(A)$ denote the subfield of
$\Qlb$ generated by $\tr(\Fr_x, A_{\bar{x}})$, where $x\in X(\F_{q^n})$ and
$n\ge 1$. Let $\sigma\colon E\to E_{\lambda'}$ be a field embedding, where
$E$ is an extension of $E(A)$. We say that $A'\in K^W(X,E_{\lambda'})$ is a
\emph{$\sigma$-companion} of $A$ if for all $x\in X(\F_{q^n})$ with $n\ge
1$, we have $\tr(\Fr_x,A'_{\bar{x}})=\sigma \tr(\Fr_x,A_{\bar{x}})$. For a
perverse Weil $\Qlb$-sheaf $\cP$, we write $E(\cP)=E([\cP])$. By a
\emph{perverse $\sigma$-companion} of $\cP$, we mean a perverse Weil
$E_{\lambda'}$-sheaf $\cP'$ such that $[\cP']$ is a $\sigma$-companion of
$[\cP]$. Proposition \ref{p.inj} implies that $\sigma$-companions in
Grothendieck groups are unique and perverse $\sigma$-companions are unique
up to semisimplification.

Let $K_\mot(X,\Qlb)$ denote the Grothendieck group of $\Perv_\mot(X,\Qlb)$
(see Remark \ref{r.Pervmot}), which is also the Grothendieck group of
$D^b_\mot(X,\Qlb)$. Then we have
\[K^W(X,\Qlb)\simeq \bigoplus_{a\in \Qlb^\times/W(q)} K_\mot(X,\Qlb)^{(a)}.\]

\begin{remark}\label{r.proj}\leavevmode
\begin{enumerate}
\item Let $A\in K^W(X,\Qlb)$ and $A'\in K^W(X,\Qlbp)$. Let $E\subseteq
    \Qlb$ and let $\sigma\colon E\to \Qlbp$ be an embedding. If for every
    embedding $\tau\colon \Qlb\to \Qlbp$ extending $\sigma$, $A'$ is a
    $\tau$-companion of $A$, then $E(A)\subseteq E$. Indeed, if $t\in
    \Qlb$ and $t'\in \Qlbp$ are such that for every $\tau$ extending
    $\sigma$, we have $\tau(t)=t'$, then $t\in E$.
\item Let $\pi_a\colon K^W(X,\Qlb)\to K_\mot(X,\Qlb)^{(a)}$ be the
    projection and let $\tau\colon \Qlb\to \Qlbp$ be an embedding. If $A'$
    is a $\tau$-companion of $A\in K^W(X,\Qlb)$, then $\pi_{\tau a} A'$ is
    a $\tau$-companion of $\pi_a A$. Indeed, for $x\in X(\F_{q^n})$, the
    set of reciprocal zeroes and roots of $\det(1-t\Fr_x,(\pi_a A)_{\bar
    x})$ is the intersection of $a^n W(q^n)$ and the set of reciprocal
    zeroes and roots of $\det(1-t\Fr_x,A_{\bar x})$, with the same
    multiplicities: if $\tr(\Fr_x, A_{\bar x})=\sum_\lambda m_\lambda
    \lambda$, then $\tr(\Fr_x, (\pi_a A)_{\bar x})=\sum_{\lambda\in a^n
    W(q^n)} m_\lambda \lambda$, so that
    \[\tau\tr(\Fr_x, (\pi_a A)_{\bar x})=\sum_{\lambda\in a^n W(q^n)}
    m_\lambda \tau\lambda=\tr(\Fr_x, (\pi_{\tau a} A')_{\bar x}).
    \]
\end{enumerate}
\end{remark}

\begin{prop}\label{p.alg}
Let $X$ be a stack and let $A\in K^W(X,\Qlb)$. The following conditions are
equivalent:
\begin{enumerate}
\item $E(A)$ is a number field.
\item $\tr(\Fr_x,A_{\bar x})$ is an algebraic number for all $n\ge 1$ and
    all $x\in X(\F_{q^n})$.
\item $A$ belongs to $K^W_\alg(X,\Qlb)\colonequals \bigoplus_{a\in
    \overline{\Q}^\times/W(q)} K_\mot(X,\Qlb)^{(a)}$.
\end{enumerate}
\end{prop}

Thus, if we identify $K^W(X,\Qlb)$ with its image under $t_X$, then
\[K^W_\alg(X,\Qlb)=\Qb^{\coprod_{n\ge 1}\lvert X(\F_{q^n})\rvert}\cap K^W(X,\Qlb).\]

\begin{proof}
It is clear that (1) implies (2). By Theorem \ref{t.gen}, (3) implies (1).
Now assume that (2) holds. Let $A=B+C$, where $B$ is the projection of $A$
in $K^W_\alg(X,\Qlb)$. Since $\det(1-t\Fr_x,A_{\bar x})\in
\overline{\Q}(t)$, we have $\det(1-t\Fr_x,C_{\bar x})=1$, so that
$\tr(\Fr_x^m,C_{\bar x})=0$ for $m\ge 1$. Thus $C=0$ by Proposition
\ref{p.inj}.
\end{proof}

For $w\in \Z$, let $K^W_w(X,\Qlb)$ denote the Grothendieck group of perverse
Weil $\Qlb$-sheaves pure of weight $w$. We have, by Remark \ref{r.mixed},
\[K^W_{\mix}(X,\Qlb)\colonequals \bigoplus_{a\in M(q)/W(q)} K_\mot(X,\Qlb)^{(a)}=\bigoplus_{w\in \Z} K^W_w(X,\Qlb).\]

\begin{remark}\label{r.sixop}
Let $f\colon X\to Y$ be a morphism of stacks, we have (bi)linear maps
\begin{gather*}
p_w\colon K^W_\mix(X,\Qlb)\to K^W_w(X,\Qlb),\\
-\otimes-,\ \cHom(-,-)\colon K^W(X,\Qlb)\times K^W(X,\Qlb)\to K^W(X,\Qlb),\\
D_X\colon K^W(X,\Qlb)\to K^W(X,\Qlb),\\
f^*,f^!\colon
K^W(Y,\Qlb)\to K^W(X,\Qlb),
\end{gather*}
where $p_w$ is the projection, $w\in \Z$. If $f$ is relatively
Deligne-Mumford, then we have linear maps
\[f_*,f_!\colon K^W(X,\Qlb)\to K^W(Y,\Qlb).\]
For an immersion of stacks $f\colon X\to Y$, the middle extension functor
$f_{!*}\colon \Perv^W(X,\Qlb)\to \Perv^W(Y,\Qlb)$ is not exact in general.
We define a linear map
\[f_{!*}\colon K^W(X,\Qlb)\to K^W(Y,\Qlb)\]
such that $f_{!*}[\cP]=[f_{!*}\cP]$ for $\cP\in \Perv^W(X,\Qlb)$ semisimple.
\end{remark}

The definition of $f_{!*}$ on Grothendieck groups is further justified by
the following fact. Let $\iota\colon \Qlb\to \C$ be an embedding and let
$w\in \R$. We let $\Perv^W_{\iota,\{w,w+1\}}(X,\Qlb)$ denote the category of
perverse Weil $\Qlb$-sheaves on $X$, $\iota$-mixed of weights $w$ and $w+1$.
Then $f_{!*}[\cP]=[f_{!*}\cP]$ for $\cP\in
\Perv^W_{\iota,\{w,w+1\}}(X,\Qlb)$ by the following immediate extension from
the case of schemes (\cite[Corollaire 2.10]{M2}, \cite[Lemma 4.1.8]{SZ}).

\begin{lemma}
Let $f\colon X\to Y$ be an immersion of stacks. The functor
\[f_{!*}\colon \Perv^W_{\iota,\{w,w+1\}}(X,\Qlb)\to\Perv^W_{\iota,\{w,w+1\}}(Y,\Qlb)\]
is exact.
\end{lemma}

We have the following generalization of theorems of Gabber.

\begin{theorem}\label{t.Gabber}
Let $X$ and $Y$ be stacks with separated diagonal. Then the operations in
Remark \ref{r.sixop} preserve $E$ and $\sigma$-companions. More precisely,
for any operation $F$ in the list and $A\in K^W(X,\Qlb)$, $E(FA)\subseteq
E(A)$, and for any $\sigma\colon E(A)\to \Qlbp$ and any $\sigma$-companion
$A'$ of $A$, $FA'$ is a $\sigma$-companion of $FA$.
\end{theorem}

By biduality $D_X D_X A=A$, it follows that $E(D_X A)=E(A)$.

\begin{proof}
The assertion on the six operations and duality is the case over a finite
field of \cite[Proposition 5.8]{Zind} (extended to Weil sheaves by Remark
\ref{r.proj} or \cite[Remarque 4.16]{Zind}), generalizing a theorem of
Gabber \cite[Theorem 2]{Gabber}. For the assertion on $p_w$ and $f_{!*}$,
where $f$ is an open immersion, we reduce to the case of separated schemes
by the existence of smooth neighborhoods \cite[Th\'eor\`eme 6.3]{LMB}. The
assertion on $p_w$ is then \cite[Proposition 2.7]{Zind} (again extended to
Weil sheaves), a consequence of Gabber's theorem on $f_{!*}$ on pure
perverse sheaves. For $f_{!*}$, by Remark \ref{r.proj} (1), it suffices to
show that $f_{!*}$ on $K^W$ preserves $\tau$-companions, where $\tau\colon
\Qlb\to \Qlbp$ is an embedding extending $\sigma$. Let $A'\in
K^W(X,\overline{\Q_{\ell'}})$ be a $\tau$-companion of $A\in K^W(X,\Qlb)$.
We have $A=\sum_{a\in \Qlb^\times/W(q)}\pi_{a} A$, $A'=\sum_{a\in
\Qlb^\times/W(q)}\pi_{\tau a} A'$, and $\pi_{\tau a} A'$ is a
$\tau$-companion of $\pi_a A$ by Remark \ref{r.proj} (2). Thus, up to
replacing $A$ by $(\pi_a A)^{(1/a_0)}$ and $A'$ by $(\pi_{\tau a}
A')^{(1/\tau a_0)}$, where $a_0\in \Qlb^\times$ is a representative of $a$,
we may assume $A,A'\in K_\mot\subseteq K^W_\mix$. Similarly, since
$A=\sum_{w\in \Z}p_w A$, $A'=\sum_{w\in \Z}p_w A'$, and $p_w$ preserves
$\tau$-companions, up to replacing $A$ by $p_w A$ and $A'$ by $p_w A'$, we
are reduced to the case when $A,A'\in K$ are pure of weight $w$, which is
Gabber's theorem \cite[Theorem~3]{Gabber}.
\end{proof}

Due to cancellation in the alternating sum, the analogues of Corollaries
\ref{c.finite2} and \ref{c.finite} do not hold: $E(A)$ is not necessarily a
finite extension of $E(f^*A)$ for $f\colon X\to Y$ surjective even for $A\in
K^W_\lisse(Y,\Qlb)$. For example, for $f\colon \Spec(\F_{q^2})\to
\Spec(\F_q)$ and $A=\left[\Qlb^{(a)}\right]-\left[\Qlb^{(-a)}\right]$, we
have $E(A)=\Q(a)$ but $E(f^*A)=\Q$.

The analogue of Proposition \ref{p.Eopen} holds for $K^W_\lisse$ with the
same proof. Moreover, Corollary \ref{c.finite2} has the following refinement
under additional assumptions.

\begin{prop}\label{p.conn} Let $f\colon X\to Y$ be a
morphism of stacks of irreducible geometric fibers. Let $A\in K^W(Y,\Qlb)$.
Then $E(f^*A)=E(A)$. Moreover, if $A'$ is a Weil $\Qlbp$-sheaf on $Y$ such
that $f^*A'$ is a $\sigma$-companion of $f^*A$ for some embedding
$\sigma\colon E(A)\to \Qlbp$, then $A'$ is a $\sigma$-companion of $A$.
\end{prop}

\begin{proof}
Recall that for any linear operator $F$ on a finite-dimensional
$(\Z/2\Z)$-graded vector space and any $N$, $\tr(F)$ can be recovered from
the numbers $\tr(F^n)$, $n\ge N$ linearly with coefficients in the field
generated by the latter (\cite[Section 1]{M2}, \cite[Lemma 8.1, Remark 8.2
(3)]{IllMisc}). Thus it suffices to show that for any $y\in Y(\F_{q^m})$,
there exists $N\ge 1$ such that for every $n\ge N$, the image of $y$ in
$Y(\F_{q^{nm}})$ lifts to $X$. This follows from the lemma below.
\end{proof}

\begin{lemma}
Let $X$ be a geometrically irreducible stack over $\F_q$. Then there exists
an integer $N\ge 1$ such that $X$ admits an $\F_{q^n}$-point for every $n\ge
N$.
\end{lemma}

\begin{proof}
Let $d$ be the dimension of $X$. Consider $H^i_c=H^i_c(X\otimes_{\F_q}
\Fqb,\Qlb)$. Then $H^{2d}_c\simeq \Qlb(-d)$, and for $j>0$, $H^{2d+j}_c=0$
and $H^{2d-j}_c$ has weights $\le 2d-\frac{j}{2}$ \cite[Theorem 1.4]{SunL}.
Let $\iota\colon \Qlb\to \C$ be an embedding. By \cite[Theorem 4.2]{SunL},
$M_n=\sum_\alpha \lvert \iota\alpha\rvert^n<\infty$, where $\alpha$ runs
through the multiset of eigenvalues of $\Fr_q$ acting on $H^{2d-j}_c$,
$j>0$. Since $M_n\le q^{(n-1)(d-\frac{1}{4})}M_1$, we have $M_n<q^{dn}$ for
$n\gg 0$. By the trace formula, we then have
\[\sum_{x\in \lvert X(\F_{q^n})\rvert}\frac{1}{\#\Aut(x)}=q^{dn}+\sum_\alpha (\pm \alpha^n)>0.\]
Here $\lvert X(\F_{q^n})\rvert$ denotes the set of isomorphism classes of
the groupoid $X(\F_{q^n})$.
\end{proof}

Finally, we deduce the existence of perverse companions and companions in
Grothendieck groups.

\begin{theorem}\label{t.perverse}
Let $X$ be a stack of separated diagonal. Let $\cP$ be a perverse Weil
$\Qlb$-sheaf on $X$. Then, for every embedding $\sigma\colon E(\cP)\to
\Qlbp$, $\cP$ admits a perverse $\sigma$-companion $\cP'$, unique up to
semisimplification. Moreover, if $E(\cP)$ is a number field, then there
exists a finite extension $E$ of $E(\cP)$ such that for every finite place
$\lambda'$ of $E$ not dividing $q$, $\cP$ admits a perverse
$\sigma_{\lambda'}$-companion. Here $\sigma_{\lambda'}\colon E(\cP)\to E \to
E_{\lambda'}$, and $E_{\lambda'}$ denotes the completion of $E$ at
$\lambda'$.
\end{theorem}

\begin{proof}
The uniqueness up to semisimplification follows from Chebotarev's density
theorem (Proposition \ref{p.inj}). For the existence of companion, we extend
$\sigma$ to an embedding $\Qlb\to \Qlbp$. If $E(\cP)$ is a number field,
then $E(\cP_i)$ is a number field for every simple factor $\cP_i$ of $\cP$
by Proposition \ref{p.alg}. Thus, for the existence of companion in both
assertions of the theorem, we may assume that $\cP$ is simple. Then $\cP$
has the form $j_{!*}(\cF[d])$ for $j\colon Y\to X$ an immersion with $Y$
smooth, $\cF$ lisse on $Y$, and $d$ the locally constant dimension function
on $Y$. The existence of companion follows from Theorem \ref{t.1} applied to
$\cF$ and the fact that $j_{!*}$ on Grothendieck groups preserves companions
(Theorem \ref{t.Gabber}).
\end{proof}

\begin{cor}
Let $X$ be a stack of separated diagonal. Let $\cP$ be a simple perverse
Weil $\Qlb$-sheaf on $X$. Then, for every embedding $\sigma\colon E(\cP)\to
\Qlbp$, there exists a unique perverse $\sigma$-companion $\cP'$. Moreover,
$\cP'$ is simple.
\end{cor}

\begin{proof}
It suffices to show the simplicity. The proof is the same as the end of the
proof of Proposition \ref{p.simple}. We extend $\sigma$ to an isomorphism
$\Qlb\simto \Qlbp$. We may assume $\cP'\simeq \bigoplus_i\cP'_i$ with
$\cP'_i$ simple. Then for each $i$ there exists a $\cP_i$ such that $\cP'_i$
is the $\sigma$-companion of $\cP_i$. It follows that $\cP\simeq
\bigoplus_i\cP_i$, so that $\cP_i=0$ for all but one $i$, and the same holds
for $\cP'_i$.
\end{proof}

\begin{cor}\label{c.K}
Let $X$ be a stack of separated diagonal. Let $\sigma\colon \Qlb\simto
\Qlbp$ be an isomorphism. For any $A\in K(X,\Qlb)$, there exists a unique
$\sigma$-companion $A'$. The map $K(X,\Qlb)\to K(X,\Qlbp)$ sending $A$ to
its $\sigma$-companion $A'$ is an isomorphism. Moreover, if $E(A)$ is a
number field, then there exists a finite extension $E$ of $E(A)$ such that
for every finite place $\lambda'$ of $E$ not dividing $q$, $A$ admits a
perverse $\sigma_{\lambda'}$-companion, where $\sigma_{\lambda'}$ is as in
Theorem \ref{t.perverse}.
\end{cor}

Note that if $A=\sum_{\cP}n_{\cP}[\cP]$, where $\cP$ runs through
isomorphism classes of simple perverse $\Qlb$-sheaves, then
$A'=\sum_{\cP}n_{\cP}[\cP']$, where $\cP'$ is the perverse
$\sigma$-companion of $\cP$, is the $\sigma$-companion of $A$.

\begin{proof}
The existence of $\sigma$-companion follows from Theorem \ref{t.perverse} or
\ref{t.1} and the uniqueness follows from Chebotarev's density theorem. For
the second assertion, note that sending $A'$ to its $\sigma^{-1}$-companion
defines an inverse of the map. For the last assertion, note that
$A=[\cP]-[\cQ]$ with $E(\cP)$ and $E(\cQ)$ being number fields by
Proposition \ref{p.alg}, so that it suffices to apply the last assertion of
the theorem.
\end{proof}

\begin{remark}
Let $X$ be a stack of separated diagonal. The group of functions
$K^W(X,\C)\subseteq \C^{\coprod_{n\ge 1}\lvert X(\F_{q^n})\rvert}$ of the
form $\iota \circ t_X(A)$, where $A$ belongs to $K^W(X,\Qlb)$ and
$\iota\colon \Qlb\simto \C$ is an isomorphism, does not depend on the choice
of $\ell$ and $\iota$ by Corollary \ref{c.K}. Similarly, the subgroups
$K_\mot(X,\Qb)\subseteq K^W_\alg(X,\Qb)\subseteq \Qb^{\coprod_{n\ge 1}\lvert
X(\F_{q^n})\rvert}$, inverse images via an embedding $i\colon \Qb\to \Qlb$
of the corresponding subgroups of $t_X(K^W(X,\Qlb))$, do not depend on the
choice of $\ell$ and $i$ (cf.\ \cite[Corollary 1.6]{Dr}).  We have
\[K^W(X,\C)\simeq \bigoplus_{a\in \C^\times/\Qb^\times} K^W_\alg(X,\Qb)^{(a)},\quad K^W_\alg(X,\Qb)\simeq \bigoplus_{a\in \Qb^\times/W(q)} K_\mot(X,\Qb)^{(a)}.\]
\end{remark}

\begin{remark}
The support of a simple perverse Weil $\Qlb$-sheaf $\cP$ on a stack $X$
equals the maximal reduced closed substack $Y$ of $X$ such that
$\tr(\Fr_x,\cP_{\bar x})=0$ for all $n\ge 1$ and all $x\in (X-Y)(\F_{q^n})$
by Proposition \ref{p.inj}. Assume that $X$ has separated diagonal. Then the
perverse $\sigma$-companion $\cP'$ of $\cP$ has the same support as $\cP$.
Sun \cite{SunSupp} defines the open support of $\cP$ to be the maximal
smooth Zariski open of $Y$ on which $\cP$ is the shift of a lisse Weil
$\Qlb$-sheaf. As he observed, $\cP$ and $\cP'$ have the same open support by
Theorem \ref{t.1}.
\end{remark}

\section{Appendix: Structure of pure perverse sheaves}\label{s.4}

The goal of this appendix is to prove the geometric semisimplicity of pure
perverse sheaves (Theorem \ref{t.decomp}).

Let $\iota\colon \Qlb\to \C$ be an embedding. Let $X$ be a stack. Let $w\in
\R$ and let $K\in D(X,\Qlb)$. We say that $K$ has \emph{$\iota$-weights $\le
w$} if the $i$-th cohomology sheaf $\cH^i K$ of $K$ has punctual
$\iota$-weights $\le w+i$ for all $i$, and $K$ has \emph{$\iota$-weights
$\ge w$} if $D K$ has $\iota$-weights $\le -w$. We say that $K$ is
\emph{$\iota$-pure} of weight $w$ if it has $\iota$-weights $\le w$ and $\ge
w$.

\begin{theorem}\label{t.decomp}
Let $X$ be a stack and let $\cP$ be an $\iota$-pure perverse Weil
$\Qlb$-sheaf on $X$. Then the pullback of $\cP$ to $X\otimes_{\F_q} \Fqb$ is
semisimple.
\end{theorem}

The case of affine stabilizers is a theorem of Sun \cite[Theorem 3.11]{Sun},
extending the case of schemes \cite[Th\'eor\`eme 5.3.8]{BBD}. Note that the
decomposition theorem of pure complexes \cite[Theorem 3.12]{Sun} does
\emph{not} extend to general stacks, as shown in \cite[Section~1]{Sun}.

As in the case of schemes \cite[Proposition 5.3.9]{BBD}, Theorem
\ref{t.decomp} has the following consequence on the structure of pure
perverse sheaves. As before we let $\cE_n$ denote the $\Qlb$-sheaf on
$\Spec(\F_q)$ of stalk $\Qlb^n$ on which $\Fr_q$ acts unipotently with one
Jordan block.

\begin{cor}
Let $X$ be a stack. The indecomposable $\iota$-pure perverse Weil
$\Qlb$-sheaves on $X$ are of the form $\cP\otimes\pi_X^*\cE_n$ with $\cP$
simple, where $\pi_X\colon X\to \Spec(\F_q)$. Moreover, for every simple
perverse Weil $\Qlb$-sheaf $\cP$, there exists a unique $m\ge 1$ such that
$\cP\simeq p_* \cQ$, where $p\colon X\otimes_{\F_q} \F_{q^m}\to X$ is the
projection, $\cQ$ is geometrically simple (i.e.\ the pullback of $\cQ$ to
$X\otimes_{\F_{q^m}}\Fqb$ is simple) and not isomorphic to any of its
conjugates under $\Gal(\F_{q^m}/\F_q)$.
\end{cor}

The first assertion of the corollary still holds with $\Qlb$ replaced by a
finite (or algebraic) extension of $\Q_\ell$.

The key to the proof of Theorem \ref{t.decomp} is a weight estimate.

\begin{prop}
Let $X$ be a stack and let $\pi\colon X\to \Spec(\F_q)$ be the projection.
Let $K\in D^{\ge 0}(X,\Qlb)$ be a complex of $\iota$-weights $\ge w$ and
vanishing $i$-th cohomology for $i<0$. Then for all $i\ge 0$, $R^i \pi_* K$
has $\iota$-weights $\ge w+\lceil\frac{i}{2}\rceil$. Moreover
$H^i(X\otimes_{\F_q} \Fqb,K)^{\Gal(\Fqb/\F_q)}=0$ for $i>0$ if $w\ge 0$, and
$R\Gamma(X,K)=0$ if $w>0$.
\end{prop}

The estimate is optimal. Indeed, for $X=BA$, where $A$ is an Abelian
variety, and $a$ of weight $1$, $R^i\pi_*(\Qlb\oplus \Qlb^{(a)}[-1])$ is
pure of weight $\lceil\frac{i}{2}\rceil$. Unlike the case of schemes or
stacks with affine stabilizers, $R^i\pi_* K$ is \emph{not} of
$\iota$-weights $\ge w+i$ in general.

\begin{proof}
The second assertion follows from the first one and the short exact sequence
\[0\to H^{i-1}(X\otimes_{\F_q} \Fqb,K)_{\Gal(\Fqb/\F_q)} \to H^i(X,K)\to H^{i}(X\otimes_{\F_q} \Fqb,K)^{\Gal(\Fqb/\F_q)}\to 0.\]
Note that for any stratification of $X$ into locally closed substacks
$(j_\alpha\colon X_\alpha\to X)_\alpha$ such that the closure of every
stratum is a union of strata, $K$ is a successive extension of
$Rj_{\alpha*}Rj_\alpha^! K$, with $Rj_\alpha^! K\in D^\ge 0$ of
$\iota$-weights $\ge w$. Thus we may assume that $X$ is smooth of dimension
$d$ and $K$ has lisse cohomology sheaves. We may further assume $K=\cF[-n]$,
with $\cF$ lisse of $\iota$-weights $\ge w+n$ and $n\ge 0$. Then the
$\iota$-weights of $(R^i \pi_* K)\spcheck\simeq (R^{2d+n-i}\pi_!\cF\spcheck)
(d)$ are at most
\[d+\frac{2d+n-i}{2}-(w+n)-2d=-w-\frac{i+n}{2}\le -w-\frac{i}{2}\]
by \cite[Theorem 1.4]{SunL}. We conclude by the fact that the
$\iota$-weights are in $w+\Z$.
\end{proof}

\begin{cor}\label{c.wt}
Let $X$ be a stack and let $\cP$ and $\cQ$ be perverse $\Qlb$-sheaves on
$X$, with $\cP$ of $\iota$-weights $\le w$, and $\cQ$ of $\iota$-weights
$\ge w$. Then for $i>0$,
$\Hom^i(\cP_{\Fqb},\cQ_{\Fqb})^{\Gal(\Fqb/\F_q)}=0$, so that the canonical
map $\Hom^i(\cP,\cQ)\to \Hom^i(\cP_{\Fqb},\cQ_{\Fqb})$ is zero. Moreover, if
$\cQ$ has $\iota$-weights $>w$, then $R\Hom(\cP,\cQ)=0$.
\end{cor}

For perverse Weil $\Qlb$-sheaves and $i=1$, the first assertion holds with
$\Hom^1$ replaced by $\Ext^1$ and $\Gal(\Fqb/\F_q)$ replaced by
$W(\Fqb/\F_q)$.

\begin{proof}
We apply the proposition to $K=R\cHom(\cP,\cQ) \in D^{\ge 0}(X,\Qlb)$, which
has $\iota$-weights $\ge 0$. If $\cQ$ has $\iota$-weights $>w$, then $K$ has
$\iota$-weights $>0$.
\end{proof}

The proof of Theorem \ref{t.decomp} is then identical to the proof of
\cite[Th\'eor\`eme 5.3.8]{BBD}, with \cite[Proposition 5.1.15]{BBD} replaced
by Corollary \ref{c.wt}.


\begin{bibdiv}
\begin{biblist}
\bib{EGAI}{article}{
   label={EGAI},
   author={Grothendieck, A.},
   title={\'El\'ements de g\'eom\'etrie alg\'ebrique. I. Le langage des sch\'emas},
   journal={Inst. Hautes \'Etudes Sci. Publ. Math.},
   number={4},
   date={1960},
   pages={5--228},
   note={R\'edig\'es avec la collaboration de J. Dieudonn\'e},
   issn={0073-8301},
   review={\MR{0217083}},
}

\bib{SGA1}{collection}{
    label={SGA1},
   title={Rev\^etements \'etales et groupe fondamental (SGA 1)},
   language={French},
   series={Documents Math\'ematiques (Paris) [Mathematical Documents (Paris)]},
   volume={3},
   note={S\'eminaire de g\'eom\'etrie alg\'ebrique du Bois Marie 1960--61.
   [Algebraic Geometry Seminar of Bois Marie 1960-61];
   Directed by A. Grothendieck;
   With two papers by M. Raynaud;
   Updated and annotated reprint of the 1971 original [Lecture Notes in
   Math., 224, Springer, Berlin;  MR0354651 (50 \#7129)]},
   publisher={Soci\'et\'e Math\'ematique de France, Paris},
   date={2003},
   pages={xviii+327},
   isbn={2-85629-141-4},
   review={\MR{2017446}},
}

\bib{SGA4}{book}{
   title={Th\'eorie des topos et cohomologie \'etale des sch\'emas},
   series={Lecture Notes in Mathematics, Vol. 269, 270, 305},
   note={S\'eminaire de G\'eom\'etrie Alg\'ebrique du Bois-Marie 1963--1964
   (SGA 4).
   Dirig\'e par M. Artin, A. Grothendieck, et J. L. Verdier. Avec la
   collaboration de N. Bourbaki, P. Deligne et B. Saint-Donat},
   publisher={Springer-Verlag},
   place={Berlin, 1972--1973},
   review={\MR{0354652 (50 \#7130)}},
   review={\MR{0354653 (50 \#7131)}},
   review={\MR{0354654 (50 \#7132)}},
   label={SGA4},
}

\bib{SGA7II}{book}{
   title={Groupes de monodromie en g\'eom\'etrie alg\'ebrique. II},
   author={Deligne, Pierre},
   author={Katz, Nicholas},
   language={French},
   series={Lecture Notes in Mathematics, Vol. 340},
   note={S\'eminaire de G\'eom\'etrie Alg\'ebrique du Bois-Marie 1967--1969 (SGA 7
   II)},
   publisher={Springer-Verlag, Berlin-New York},
   date={1973},
   pages={x+438},
   review={\MR{0354657}},
   label={SGA7II},
}

\bib{Abe}{article}{
    author={Abe, Tomoyuki},
    title={Langlands correspondence for isocrystals and existence of crystalline companion for curves},
    journal={J. Amer.\ Math.\ Soc.},
    doi={10.1090/jams/898},
}

\bib{Behrend}{article}{
   author={Behrend, Kai A.},
   title={Derived $l$-adic categories for algebraic stacks},
   journal={Mem. Amer. Math. Soc.},
   volume={163},
   date={2003},
   number={774},
   pages={viii+93},
   issn={0065-9266},
   review={\MR{1963494 (2004e:14006)}},
   doi={10.1090/memo/0774},
}

\bib{BBD}{article}{
   author={Be{\u\i}linson, A. A.},
   author={Bernstein, J.},
   author={Deligne, P.},
   title={Faisceaux pervers},
   language={French},
   conference={
      title={Analysis and topology on singular spaces, I},
      address={Luminy},
      date={1981},
   },
   book={
      series={Ast\'erisque},
      volume={100},
      publisher={Soc. Math. France},
      place={Paris},
   },
   date={1982},
   pages={5--171},
   review={\MR{751966 (86g:32015)}},
}

\bib{BR}{article}{
   author={B\'enabou, Jean},
   author={Roubaud, Jacques},
   title={Monades et descente},
   language={French},
   journal={C. R. Acad. Sci. Paris S\'er. A-B},
   volume={270},
   date={1970},
   pages={A96--A98},
   review={\MR{0255631}},
}

\bib{BAlg}{book}{
   author={Bourbaki, N.},
   title={\'El\'ements de math\'ematique. Alg\`ebre. Chapitres 4 \`a 7},
   language={French},
   publisher={Springer},
   date={2007},
   note={R\'eimpression de l'\'edition de 1981},
   label={B1},
}

\bib{BAlg8}{book}{
   author={Bourbaki, N.},
   title={\'El\'ements de math\'ematique. Alg\`ebre. Chapitre 8. Modules et anneaux
   semi-simples},
   language={French},
   note={Second revised edition of the 1958 edition [MR0098114]},
   publisher={Springer, Berlin},
   date={2012},
   pages={x+489},
   isbn={978-3-540-35315-7},
   isbn={978-3-540-35316-4},
   review={\MR{3027127}},
   doi={10.1007/978-3-540-35316-4},
   label={B2},
}

\bib{Chin}{article}{
   author={Chin, CheeWhye},
   title={Independence of $\ell$ in Lafforgue's theorem},
   journal={Adv. Math.},
   volume={180},
   date={2003},
   number={1},
   pages={64--86},
   issn={0001-8708},
   review={\MR{2019215 (2005b:14041)}},
   doi={10.1016/S0001-8708(02)00082-8},
}

\bib{CR}{book}{
   author={Curtis, Charles W.},
   author={Reiner, Irving},
   title={Representation theory of finite groups and associative algebras},
   series={Pure and Applied Mathematics, Vol. XI},
   publisher={Interscience Publishers, a division of John Wiley \& Sons, New
   York-London},
   date={1962},
   pages={xiv+685},
   review={\MR{0144979}},
}

\bib{Delignec}{article}{
   author={Deligne, Pierre},
   title={Les constantes des \'equations fonctionnelles des fonctions $L$},
   language={French},
   conference={
      title={Modular functions of one variable, II},
      address={Proc. Internat. Summer School, Univ. Antwerp, Antwerp},
      date={1972},
   },
   book={
      publisher={Springer, Berlin},
   },
   date={1973},
   pages={501--597. Lecture Notes in Math., Vol. 349},
   review={\MR{0349635}},
}

\bib{WeilII}{article}{
   author={Deligne, Pierre},
   title={La conjecture de Weil. II},
   language={French},
   journal={Inst. Hautes \'Etudes Sci. Publ. Math.},
   number={52},
   date={1980},
   pages={137--252},
   issn={0073-8301},
   review={\MR{601520 (83c:14017)}},
}

\bib{Deligne}{article}{
   author={Deligne, Pierre},
   title={Finitude de l'extension de $\mathbb{Q}$ engendr\'ee par des traces de
   Frobenius, en caract\'eristique finie},
   language={French, with French and Russian summaries},
   journal={Mosc. Math. J.},
   volume={12},
   date={2012},
   number={3},
   pages={497--514, 668},
   issn={1609-3321},
   review={\MR{3024820}},
}

\bib{Dr}{article}{
   author={Drinfeld, Vladimir},
   title={On a conjecture of Deligne},
   language={English, with English and Russian summaries},
   journal={Mosc. Math. J.},
   volume={12},
   date={2012},
   number={3},
   pages={515--542, 668},
   issn={1609-3321},
   review={\MR{3024821}},
}

\bib{DK}{article}{
    author={Drinfeld, Vladimir},
    author={Kedlaya, Kiran},
    title={Slopes of indecomposable $F$-isocrystals},
    note={arXiv:1604.00660v9, preprint},
    year={2018},
}

\bib{EK}{article}{
   author={Esnault, H{\'e}l{\`e}ne},
   author={Kerz, Moritz},
   title={A finiteness theorem for Galois representations of function fields
   over finite fields (after Deligne)},
   journal={Acta Math. Vietnam.},
   volume={37},
   date={2012},
   number={4},
   pages={531--562},
   issn={0251-4184},
   review={\MR{3058662}},
}

\bib{Gabber}{article}{
   author={Fujiwara, Kazuhiro},
   title={Independence of $l$ for intersection cohomology (after Gabber)},
   conference={
      title={Algebraic geometry 2000, Azumino (Hotaka)},
   },
   book={
      series={Adv. Stud. Pure Math.},
      volume={36},
      publisher={Math. Soc. Japan},
      place={Tokyo},
   },
   date={2002},
   pages={145--151},
   review={\MR{1971515 (2004c:14038)}},
}

\bib{Giraud}{article}{
   author={Giraud, Jean},
   title={M\'ethode de la descente},
   language={French},
   journal={Bull. Soc. Math. France M\'em.},
   volume={2},
   date={1964},
   pages={viii+150},
   review={\MR{0190142}},
}

\bib{IllMisc}{article}{
   author={Illusie, Luc},
   title={Miscellany on traces in $\ell$-adic cohomology: a survey},
   journal={Jpn. J. Math.},
   volume={1},
   date={2006},
   number={1},
   pages={107--136},
   issn={0289-2316},
   review={\MR{2261063}},
   doi={10.1007/s11537-006-0504-3},
}

\bib{IZ}{article}{
   author={Illusie, Luc},
   author={Zheng, Weizhe},
   title={Quotient stacks and equivariant \'etale cohomology algebras:
   Quillen's theory revisited},
   journal={J. Algebraic Geom.},
   volume={25},
   date={2016},
   number={2},
   pages={289--400},
   issn={1056-3911},
   review={\MR{3466353}},
}

\bib{Knutson}{book}{
   author={Knutson, Donald},
   title={Algebraic spaces},
   series={Lecture Notes in Mathematics, Vol. 203},
   publisher={Springer-Verlag, Berlin-New York},
   date={1971},
   pages={vi+261},
   review={\MR{0302647}},
}

\bib{Lafforgue}{article}{
   author={Lafforgue, Laurent},
   title={Chtoucas de Drinfeld et correspondance de Langlands},
   language={French, with English and French summaries},
   journal={Invent. Math.},
   volume={147},
   date={2002},
   number={1},
   pages={1--241},
   issn={0020-9910},
   review={\MR{1875184 (2002m:11039)}},
   doi={10.1007/s002220100174},
}

\bib{VLaff}{article}{
   author={Lafforgue, Vincent},
   title={Estim\'ees pour les valuations $p$-adiques des valeurs propes des
   op\'erateurs de Hecke},
   language={French, with English and French summaries},
   journal={Bull. Soc. Math. France},
   volume={139},
   date={2011},
   number={4},
   pages={455--477},
   issn={0037-9484},
   review={\MR{2869300}},
}

\bib{LaumonTF}{article}{
   author={Laumon, G{\'e}rard},
   title={Transformation de Fourier, constantes d'\'equations fonctionnelles
   et conjecture de Weil},
   language={French},
   journal={Inst. Hautes \'Etudes Sci. Publ. Math.},
   number={65},
   date={1987},
   pages={131--210},
   issn={0073-8301},
   review={\MR{908218 (88g:14019)}},
}


\bib{LMB}{book}{
   author={Laumon, G{\'e}rard},
   author={Moret-Bailly, Laurent},
   title={Champs alg\'ebriques},
   series={Ergebnisse der Mathematik und ihrer Grenzgebiete. 3. Folge. A
   Series of Modern Surveys in Mathematics},
   volume={39},
   publisher={Springer-Verlag},
   place={Berlin},
   date={2000},
   pages={xii+208},
   isbn={3-540-65761-4},
   review={\MR{1771927 (2001f:14006)}},
}

\bib{Mumford}{book}{
   author={Mumford, David},
   title={Abelian varieties},
   series={Tata Institute of Fundamental Research Studies in Mathematics},
   volume={5},
   note={With appendices by C. P. Ramanujam and Yuri Manin;
   Corrected reprint of the second (1974) edition},
   publisher={Published for the Tata Institute of Fundamental Research,
   Bombay; by Hindustan Book Agency, New Delhi},
   date={2008},
   pages={xii+263},
   isbn={978-81-85931-86-9},
   isbn={81-85931-86-0},
   review={\MR{2514037 (2010e:14040)}},
}

\bib{Noohi}{article}{
   author={Noohi, B.},
   title={Fundamental groups of algebraic stacks},
   journal={J. Inst. Math. Jussieu},
   volume={3},
   date={2004},
   number={1},
   pages={69--103},
   issn={1474-7480},
   review={\MR{2036598}},
   doi={10.1017/S1474748004000039},
}

\bib{Rydh}{article}{
   author={Rydh, David},
   title={Submersions and effective descent of \'etale morphisms},
   language={English, with English and French summaries},
   journal={Bull. Soc. Math. France},
   volume={138},
   date={2010},
   number={2},
   pages={181--230},
   issn={0037-9484},
   review={\MR{2679038}},
}

\bib{SP}{book}{
  author       = {The Stacks Project Authors},
  title        = {Stacks Project},
  note         = {\url{http://stacks.math.columbia.edu}},
  label ={SP},
}

\bib{SunL}{article}{
   author={Sun, Shenghao},
   title={$L$-series of Artin stacks over finite fields},
   journal={Algebra Number Theory},
   volume={6},
   date={2012},
   number={1},
   pages={47--122},
   issn={1937-0652},
   review={\MR{2950161}},
   doi={10.2140/ant.2012.6.47},
}

\bib{Sun}{article}{
   author={Sun, Shenghao},
   title={Decomposition theorem for perverse sheaves on Artin stacks over
   finite fields},
   journal={Duke Math. J.},
   volume={161},
   date={2012},
   number={12},
   pages={2297--2310},
   issn={0012-7094},
   review={\MR{2972459}},
   doi={10.1215/00127094-1723657},
}

\bib{SunSupp}{article}{
    author={Sun, Shenghao},
    title={Independence of $\ell$ for the supports in the decomposition theorem},
    note={arXiv:1512.01447v1, preprint},
    year={2015},
}

\bib{SZ}{article}{
   author={Sun, Shenghao},
   author={Zheng, Weizhe},
   title={Parity and symmetry in intersection and ordinary cohomology},
   journal={Algebra Number Theory},
   volume={10},
   date={2016},
   number={2},
   pages={235--307},
   issn={1937-0652},
   review={\MR{3477743}},
   doi={10.2140/ant.2016.10.235},
}

\bib{Voe}{article}{
   author={Voevodsky, V.},
   title={Homology of schemes},
   journal={Selecta Math. (N.S.)},
   volume={2},
   date={1996},
   number={1},
   pages={111--153},
   issn={1022-1824},
   review={\MR{1403354}},
   doi={10.1007/BF01587941},
}

\bib{M2}{article}{
   author={Zheng, Weizhe},
   title={Th\'eor\`eme de Gabber d'ind\'ependance de $l$},
   language={French},
   note={M\'emoire de Master deuxi\`eme ann\'ee, \href{http://arXiv.org/abs/1608.06191}{arXiv:1608.06191}},
   year={2005},
}
		
\bib{Zint}{article}{
   author={Zheng, Weizhe},
   title={Sur la cohomologie des faisceaux $l$-adiques entiers sur les corps
   locaux},
   language={French, with English and French summaries},
   journal={Bull. Soc. Math. France},
   volume={136},
   date={2008},
   number={3},
   pages={465--503},
   issn={0037-9484},
   review={\MR{2415350 (2009d:14015)}},
}

\bib{Zind}{article}{
   author={Zheng, Weizhe},
   title={Sur l'ind\'ependance de $l$ en cohomologie $l$-adique sur les
   corps locaux},
   language={French, with English and French summaries},
   journal={Ann. Sci. \'Ec. Norm. Sup\'er. (4)},
   volume={42},
   date={2009},
   number={2},
   pages={291--334},
   issn={0012-9593},
   review={\MR{2518080 (2010i:14032)}},
}


\end{biblist}
\end{bibdiv}
\end{document}